\documentclass[11pt]{amsart}

\usepackage{amsmath}
\usepackage{amssymb}
\usepackage{graphicx}

\newtheorem{theorem}{Theorem}[section]
\newtheorem{corollary}[theorem]{Corollary}
\newtheorem{lemma}[theorem]{Lemma}
\newtheorem{proposition}[theorem]{Proposition}
\theoremstyle{definition}
\newtheorem{definition}[theorem]{Definition}
\newtheorem{example}[theorem]{Example}
\newtheorem{remark}[theorem]{Remark}

\numberwithin{equation}{section}

\def\A{{\mathcal {A}}}
\def\P{{\mathcal P}}
\def\mK{{\mathcal{K}}}
\def\S{{\mathcal {S}}}
\def\V{{\mathcal {V}}}
\def\K{{\mathcal {K}}}

\newcommand{\R}{\mathbb{R}}

\newcommand{\C}{\mathbb{C}}
\newcommand{\N}{\mathbb{N}}

\newcommand{\T}{\mathcal{T}}

\renewcommand{\d}{\, \mathrm{d}}

\newcommand{\lf}{\left}
\newcommand{\rt}{\right}
\newcommand{\vphi}{\varphi}
\newcommand{\no}{\nonumber}

\newcommand{\as}[1]{\begin{align*}#1\end{align*}}
\newcommand{\al}[1]{\begin{align}#1\end{align}}

\newcommand{\lbl}[1]{\label{#1}}

\newcommand{\eq}[1]{\begin{equation}#1\end{equation}}
\newcommand{\eqs}[1]{\begin{equation*}#1\end{equation*}}

\newcommand{\refn}[1]{(\ref{#1})}

\newcommand{\mcal}[1]{\mathcal{#1}}
\newcommand{\cal }{\mathcal }

\newcommand{\thm}[1]{\begin{theorem}#1\end{theorem}}

\newcommand{\coro}[1]{\begin{corollary}#1\end{corollary}}
\newcommand{\rem}[1]{\begin{remark}#1\end{remark}}
\newcommand{\prop}[1]{\begin{proposition}#1\end{proposition}}
\newcommand{\lem}[1]{\begin{lemma}#1\end{lemma}}
\newcommand{\prf}[1]{\begin{proof}#1\end{proof}}
\newcommand{\defn}[1]{\begin{definition}#1\end{definition}}

\begin{document}
\title[Linear transfers as minimal costs of dilations of measures]{Linear transfers as minimal costs of dilations of measures in balayage order}



\author{Nassif  Ghoussoub}
\address{Department of Mathematics,  The University of British Columbia
Vancouver BC Canada V6T 1Z2}
\email{\tt nassif@math.ubc.ca}
\date{December 11, 2022, revised March 03, 2023}



\maketitle
\centerline{Dedicated to Professor David Preiss on the occasion of his 75th birthday}
\begin{abstract}  Linear transfers between probability distributions were  introduced in \cite{BG1, BG2} in order to extend the theory of optimal mass transportation while preserving the important duality established by Kantorovich. It is shown here that $\{0, +\infty\}$-valued linear transfers can be characterized by balayage of measures with respect to suitable cones of functions \`a la Choquet, while  general linear transfers extend balayage theory by requiring the ``sweeping out" of measures to optimize certain cost functionals. We study the dual class of Kantorovich operators, which are natural and manageable extensions of Markov operators. It is also an important subclass of capacities, and could be called ``convex functional Choquet capacities," since they play for non-linear maps the same role that convex envelopes do for arbitrary numerical functions. A forthcoming paper \cite{BG3} will study their ergodic properties and their applications. 

  \end{abstract}
  
  \tableofcontents

\section{Introduction}
Given a closed convex cone $\mathcal{A}$ of lower semi-continuous functions on a compact metric space $\Omega$ that is stable under finite maxima and containing the non-negative constant functions, one defines a  partial order on the set $\P(\Omega)$ of probability measures via the relation
\begin{equation}\label{true}
\text{$\mu \preceq_{\mathcal{A}} \nu$\,\, if \, $\int_{\Omega} \phi d\mu \leq \int_{\Omega} \phi d\nu$ for all $\phi \in \mathcal{A}$.}
\end{equation} 
This is the so-called {\em balayage order with respect to $\mathcal{A}$} and cones with the above properties will be called {\em balayage cones}. They often appear in convex analysis, probability theory, analysis in several complex variables and potential theory. Here are a few examples. 
\begin{itemize}
\item The cone of lower semi-continuous convex functions on a convex compact set $\Omega$ in a locally convex topological vector space.  
\item The cone of subharmonic (resp., plurisubharmonic) functions on a domain $\Omega$ in $\R^n$ (resp., ${\mathbb C}^n$). 
\item The cone of excessive (resp., supermedian) functions associated to a homogenous Markov process. 
\item The cone of $1$-Lipschitz functions on a compact metric space $\Omega$, as well as its tangent cone at a given $f$, i.e., 
$\A_f=\{\lambda (f-g); \, \lambda \geq 0, g \, \hbox{is $1$-Lipschitz on  $\Omega$}\}$.  
\end{itemize}
 We are interested in balayage as a transport procedure between probability distributions that are sometimes supported on different compact metric spaces $X$ and $Y$. We shall then consider closed convex cones $\mathcal{A}$ of lower semi-continuous  functions on the disjoint union $\Omega=X \sqcup Y$ and define a {\em restricted balayage} only on pairs of probabilities $(\mu, \nu)$ in $\mathcal{P}(X)\times \mathcal{P}(Y)$ -as opposed to $\P(X \sqcup Y)$- via: 
\begin{equation}\label{untrue}
\text{$\mu \preceq_{\mathcal{A}} \nu$\, if and only if\, $\int_{X} \phi d\mu \leq \int_{Y} \phi d\nu$ for all $\phi \in \mathcal{A}$.}
\end{equation} 
Note the difference between (\ref{true}) and (\ref{untrue}). For example, the first pre-order is a transitive relation while the second is not, even if $X=Y$. Still, a celebrated theorem of Strassen \cite{St} can be applied whenever (\ref{untrue}) holds to deduce the following (See for example \cite{C}). 

\begin{proposition}
For every pair $(\mu,\nu)$ in $\mathcal{P}(X)\times \mathcal{P}(Y)$ with $\mu \preceq_{\A}\nu$, where  $\mathcal{A}$ is a balayage  cone on $X \sqcup Y$, there exists a Markov kernel (or a transport plan) $\pi \in \mathcal{P}(X\times Y)$ whose marginals are $(\mu,\nu)$ and such that the disintegration $(\pi_x)_{x\in X}$ of $\pi$ with respect to $\mu$, i.e,
\begin{equation}
\pi (X\times B)=\nu(B)\,\, {\rm and}\,\, \pi (A\times B)=\int_A\pi_x(B) \, d\mu (x) \hbox{ for all Borel $A\subset X$, $B\subset Y$},
\end{equation}
also satisfies 
\begin{equation}\delta_x \preceq_{\A} \pi_x \quad \mu-{\rm a.e.}
\end{equation}
\end{proposition} 
 These are sometimes called {\it $\A$-dilations} of $\mu$ into $\nu$. The set of such $\A$-dilations will be denoted $\mK_\A(\mu, \nu)$. The extreme points of the set $\{(\mu,\nu)\,;\, \mu \preceq_{\A} \nu\}$ are then pairs of the form $(\delta_x, \eta)$ where $\delta_x \preceq_{\mathcal{A}} \eta$.  

In this paper, we consider the problem of finding $\A$-dilations of $\mu$ into $\nu$ that minimize appropriate cost functionals of the form $c: X\times {\mathcal P}(Y)\to  \R \cup\{+\infty\}$, that is 
\begin{equation}\label{optimalbalayage1}
\T(\mu,\nu) = 
\inf\{\int_{X} c(x, \pi_x)d\mu(x)\,;\, \pi \in \mathcal{K}_{\A}(\mu,\nu)\}.
\end{equation}
This can be seen as a problem of finding an optimal Strassen disintegration of transport plans between two probability measures in balayage order. This is known as the {\it optimal martingale transport} \cite{Hl, GKL} whenever $\A$ is the cone of convex functions, and $c(x, \sigma)=\int_Yd(x, y)\, d\sigma(y)$ with $d(x,y)$ being the  cost of moving mass from $x$ to $y$. 

One of our objectives is to characterize those convex functions $\T$ on 
 $ {\mathcal P}(X)\times {\mathcal P}(Y)$ that can be described as a value function of an optimal transport problem as in (\ref{optimalbalayage1}) for a suitable cost function $c$ and an appropriate balayage cone $\A$. These turn out to be the \textit{backward linear transfers} between probability distributions, a notion we introduced in \cite{BG1, BG2} in order to extend the theory of optimal mass transportation while preserving the important duality established by Kantorovich. 
 
 \textit{Backward linear transfers} are essentially lower semi-continuous convex functionals $\T$ on $\mathcal{P}(X)\times \mathcal{P}(Y)$ whose partial maps $\T_\mu:\nu \mapsto \T(\mu,\nu)$ are such that their Legendre-Fenchel transforms lead to a map $\mu \to \T_\mu^*$ that is \textit{linear} with respect to $\mu$. This class of functionals is rich enough to naturally extend the cone of convex lower semi-continuous energies on Wasserstein space, Markov operators, and cost minimizing mass transports, but also many other couplings between probability measures to which Monge-Kantorovich theory does not readily apply. Examples include 
{\it optimal martingale transports} \cite{Hl, GKL},  {\it optimal Skorokhod embeddings} \cite{GKP1, GKP2}, optimal stochastic transports \cite{M-T, DGKP, GKP3} and the {\it weak mass transports} of Talagrand \cite{Ta}, Marton \cite{Ma}, Gozlan et al. \cite{Go}. The class also includes various couplings such as the {\it Schr\"odinger bridge} associated to a reversible Markov process \cite{GLR}. Here is the precise definition:
 
 Let ${\cal M}(X)$ is the class of signed Radon measures on $X$. The domain of a functional $\T:\mcal{M}(X) \times \mcal{M}(Y) \to \R\cup \{+\infty\}$ will be denoted 
 \eq{D(\T):=\{(\mu, \nu) \in {\mathcal M}(X)\times  {\mathcal M}(Y);  \T (\mu, \nu)<+\infty\}.
}
 We consider for each $\mu \in {\mathcal P}(X)$ (resp.,  $\nu \in {\mathcal P}(Y)$), the partial maps ${\mathcal T}_\mu$ on ${\mathcal P}(Y)$ (resp., ${\mathcal T}_\nu$ on ${\mathcal P}(X)$) given by $\nu \to {\mathcal T} (\mu, \nu)$ (resp., $\mu \to {\mathcal T} (\mu, \nu)$). 
 
  We let $C(X)$ (resp., $B(X)$) (resp., $USC(X)$) (resp., $LSC(X)$) be  
 the space of continuous (resp., Borel measurable) (resp., bounded above, proper and upper semi-continuous), (resp., bounded below, proper and lower semi-continuous) functions on a compact space $X$.

\defn{\lbl{lineartransfers} A functional $\T:\mcal{M}(X) \times \mcal{M}(Y) \to \R\cup \{+\infty\}$ 
 is said to be a \textbf{backward linear transfer} (resp., \textbf{forward linear transfer}) if 
\begin{enumerate}
\item $\T:\mcal{M}(X) \times \mcal{M}(Y) \to \R\cup \{+\infty\}$ is a proper, convex, bounded below, and weak$^*$ lower semi-continuous. 
\item $D(\T) \subset \P(X)\times \P(Y)$.
\item There exists a map $T^-$ (resp., $T^+$) from $C(Y)$ (resp., $C(X)$) into the space of bounded above (resp., below) Borel-measurable functions on $X$ (resp., on $Y$) such that for all $\mu \in \P(X)$  
and $g \in C(Y)$ (resp., $\nu\in \P(Y)$  and $f \in C(X)$), 
\eq{\lbl{Legendre_trans}
\T_\mu^*(g) = \int_{X}T^- g d\mu,\quad ({\rm resp.},\,\,  \T_\nu^*(f) = -\int_{Y}T^+ (-f) d\nu),
}
where $\T_\mu^*$ (resp., $\T_\nu^*$) is the Fenchel-Legendre transform of $\T_\mu$ (resp., $\T_\nu$) with respect to the duality of $C(X)$ and ${\mathcal M}(X)$ (resp.,   $C(Y)$ and ${\mathcal M}(Y))$). 
\end{enumerate}
Note that since $D(\T) \subset \P(X)\times \P(Y)$, the Legendre transforms are simply
\begin{equation}
\T_\mu^*(g) := \sup_{\sigma \in \mathcal{P}(Y)}\{\int_{Y}g d\sigma - \T(\mu, \sigma)\}\quad {\rm resp.,}\quad \T_\nu^*(f) := \sup_{\sigma \in \mathcal{P}(X)}\{\int_{X}f d\sigma - \T(\sigma, \nu)\}.
\end{equation}

}
The map $T^-$ (resp., $T^+$) is then called the {\bf backward (resp., forward) Kantorovich operator} associated to the transfer $\T$. We shall restrict our analysis to backward transfers since if $\T$ is a forward linear transfer with $T^+$ as a forward Kantorovich operator, then $\tilde\T(\mu, \nu)=\T(\nu, \mu)$ is a backward linear transfer with $T^-f=-T^+(-f)$ being the corresponding backward Kantorovich operator. 

Note that since $T^-$ and $T^+$ arise from a Legendre transform, they must satisfy certain additional properties  that we exhibit in Section 2. These -mostly non-linear- Kantorovich operators are important extensions of \textit{Markov operators} and are ubiquitous in mathematical analysis and differential equations. They appear as the \textit{r\'eduite operators} in potential theory, as  \textit{filling scheme operators} in ergodic theory, as well as in Monge-Kantorovich duality of  optimal mass transport (hence the name). They also include maps that associate to an initial (resp., final) state the solution at a given time $t$ (resp., initial time) of a first or second-order \textit{Hamilton-Jacobi equation}. They are actually general value functions in \textit{dynamic programming principles} (\cite{FS} Section II.3) and also appear in the mathematical theory of \textit{image processing} \cite{AGL}. 

Linear transfers that can only take $0$ or $+\infty$ values are particularly interesting since we shall be able to characterize them completely in terms of the classical notion of balayage. 
 
 \defn{\lbl{transfer.sets_defn} Say that a subset ${\mathcal S}$ of ${\mathcal P}(X)\times {\mathcal P}(Y)$ is a {\em backward transfer set} if its characteristic function 
\begin{equation}
\T(\mu,\nu) := \begin{cases} 0 & \text{if } (\mu,\nu) \in \S,\\
+\infty & \text{otherwise},
\end{cases}
\end{equation}
is a backward linear transfer. 
 }
\defn{\lbl{standard} 1) Say that a functional $\T:\mcal{P}(X) \times \mcal{P}(Y) \to \R\cup \{+\infty\}$ is {\bf standard} if  
\begin{equation} 
\hbox{for each $x\in X$, there is $\sigma \in \P(Y)$ with $\T(\delta_x, \sigma)<+\infty$.}
\end{equation} 

2) A set $\S\subset \mcal{P}(X) \times \mcal{P}(Y)$ is then said to be {\bf standard}\footnote{It turned out that, motivated by some work of Dubins and Savage,  Dellacherie-Meyer \cite{D-M} had called such sets {\bf gambling houses.} We shall adopt this terminology throughout this paper.} if its characteristic function is {\em standard}, that is if for each $x\in X$, there is $\sigma \in \P(Y)$ with $(\delta_x, \sigma)\in \S$.
 }

A typical example of a transfer set  is given by the classical notion of convex order.\\

\noindent{\it Balayage of measures in convex order:} If $X$ is a convex compact subset of a locally convex topological vector space, consider the functional  
 \begin{equation}
{\mathcal B}(\mu, \nu)=\left\{ \begin{array}{llll}
0 \quad &\hbox{if $\mu \prec \nu$}\\
+\infty \quad &\hbox{\rm otherwise,}
\end{array} \right.
\end{equation} 
and where 
 $
 \mu \prec \nu$ if and only if  $\int_X\phi \, d\mu \leq \int_X\phi \, d\nu$ for all convex continuous functions $\phi$. It follows from standard Choquet theory \cite{Ch2}, that ${\mathcal B}$ is a backward linear transfer whose Kantorovich operator is $T^-f={\hat f}$, where ${\hat f}$ is the upper semi-continuous concave upper envelope of $f$. It is also a forward  linear transfer with $T^+f$ being the lower semi-continuous convex lower envelope of $f$. In other words,  
 \begin{equation}\lbl{bal.set}
 \S=\{(\mu, \nu) \in \P(X)\times \P(X);\, \mu \prec \nu\} 
\end{equation} 
 is a standard transfer set. 
 
Balayage orders (though ``restricted") with respect to appropriate cones turned out to characterize transfer sets, as we shall prove in Section 3 the following. For that, recall that  a  {\em positively  $1$-homogenous operator}  $T$ is one that verifies $T(\lambda f)=\lambda Tf$ for any $\lambda \geq 0$.

 \begin{theorem} \label{transfer.set} Let ${\mcal S}$ be a  weak$^*$-compact convex subset of $\mcal{P}(X)\times\mcal{P}(Y)$.
 The following are then equivalent:
\begin{enumerate}

\item ${\mcal S}$ is a gambling house. 

\item ${\mcal S}$ is a standard backward transfer set.

\item  $\S=\{(\mu, \nu) \in \P(X)\times \P(Y);\, \mu \prec_{\mcal A} \nu\}$, where $\mcal A$  is a  balayage cone on  $X \sqcup Y$.

\item 
$(\mu, \nu) \in {\mcal S}$ if and only if there exists $\pi \in {\mathcal K}(\mu, \nu)$ such that  
$(\delta_x, \pi_x) \in {\mcal S}$ for $\mu$-almost $x\in X$.
\item $\S=\{(\mu, \nu) \in \P(X)\times \P(Y);\, \nu \leq T_\#\mu \}$, where $T:C(Y)\to USC(X)$
is a positively $1$-homogenous Kantorovich operator
and $T_\#\mu (g):=\int_XTg \, d\mu$ for every $g\in C(Y)$. 
\end{enumerate}

\end{theorem}
 \begin{definition} Sets of the form (\ref{bal.set}) will be called {\bf restricted (resp., true) balayage sets} if the cone $\A$ is in $LSC(X \sqcup Y)$ (resp., if $X=Y$ and $\A\subset LSC(X)$). 
\end{definition}

Note the two basic differences between an {\em extended balayage} and a {\em true balayage}, even when $X=Y$. For one, a probability measure is comparable to itself in a true balayage order, that is the set of  ordered pairs must containing the diagonal $\{(\mu, \mu); \mu\in {\mathcal P}(X)\}$. Another difference is that a true balayage is clearly transitive, that is if $\mu\prec_\A \sigma$ and $\sigma\prec_\A \nu$ then $\mu\prec_\A \nu$. Neither property makes sense for an extended balayage, even if $X=Y$. We prove in Section 4 that these two properties actually suffice to have a backward transfer set represented by a true balayage.  
 \defn{\lbl{transitive} Say that a subset ${\mcal S}$ of $\mcal{P}(X)\times\mcal{P}(X)$ is {\bf transitive} if 
 \begin{equation}
 (\mu, \sigma)\in \S \,\, {\rm and}\,\, (\sigma, \nu)\in \S \Rightarrow (\mu, \nu)\in \S.
 \end{equation}
 }

\begin{theorem}\label{true.balayage0} Let ${\mcal S}$ be a subset of $\mcal{P}(X)\times\mcal{P}(X)$. Then the following are equivalent:
\begin{enumerate}
\item ${\mcal S}$ is a transitive backward transfer set containing the diagonal $\{(\mu, \mu); \mu\in {\mathcal P}(X)\}$. 
\item $\S=\{(\mu, \nu) \in \P(X)\times \P(X);\, \mu \prec_{\mcal A} \nu\}$, where $\A$ is a balayage cone in $C(X)$. 
\item $\S=\{(\mu, \nu) \in \P(X)\times \P(X);\, \nu \leq T_\#\mu  \}$, where $T:C(X)\to USC(X)$ is an idempotent positively $1$-homogenous Kantorovich operator  such that $Tf\geq f$ for all $f\in C(Y)$.
\end{enumerate}
\end{theorem}
 Actually, we shall show that any standard backward transfer set  is contained in a minimal transfer set represented by a  true balayage. 
 
 The following example illustrates the difference between these two notions of balayage.\\
  
 \noindent{\it The prescribed Markov transfer:}\lbl{mark}
    Let $T: C(X) \to C(X)$ be a bounded linear positive operator such that $T1=1$, i.e., a Markov operator, then one can associate a backward linear transfer in the following way:
\begin{equation}\lbl{Markov}
{\mathcal T}_T(\mu, \nu) :=\left\{ \begin{array}{llll}
0 \quad &\hbox{if $\nu = T^*(\mu) $}\\
+\infty \quad &\hbox{\rm otherwise,}
\end{array} \right.
\end{equation}
 where $T^*:{\mathcal M}(X) \to {\mathcal M}(X)$ is the adjoint operator. This is an example of a transfer set which cannot arise from a true balayage with respect to a cone $\A$ in $C(X)$, unless $T$ is the identity. On the other hand, we shall see in Section 4 that it is associated to a restricted balayage, and that the smallest true balayage set containing it is the one corresponding to balayage with respect to the cone of $T$-superharmonic functions.

So far, we have only encountered  positively  $1$-homogenous Kantorovich operators. The next examples are not. \\

\noindent{\it Optimal mass transport:} Cost minimizing mass transports \` a la Monge-Kantorovich \cite{Vi} are functionals on ${\mathcal P}(X)\times {\mathcal P}(Y)$ of the form,
\begin{eqnarray}
{\mathcal T}_c(\mu, \nu):=\inf\big\{\int_{X\times Y} c(x, y) \, d\pi; \pi\in \mK(\mu,\nu)\big\},
\end{eqnarray}
where $c(x, y)$ is a lower semi-continuous cost function on the product measure space $X\times Y$, and $\mK(\mu,\nu)$ is the set of probability measures $\pi$ on $X\times Y$ whose marginal on $X$ (resp. on $Y$) is $\mu$ (resp., $\nu$) {\it (i.e., the transport plans)}. A consequence of the Monge-Kantorovich theory is that cost minimizing transports ${\mathcal T}_c$ are both forward and backward linear transfers with Kantorovich operators given for any $f\in C(X)$ (resp., $g\in C(Y)$), by 
\begin{equation}
T ^+_cf(y)=\inf_{x\in X} \{c(x, y)+f(x)\} \quad {\rm and} \quad T ^-_cg(x)=\sup_{y\in Y} \{g(y)-c(x, y)\}. 
\end{equation}

\noindent{\it Optimal weak mass transport:} A typical example of a backward linear transfer is the following natural generalization of optimal transport: The {\bf optimal weak transport}, formally introduced by Gozlan et. al. \cite{Go}, in order to include previous work by Talagrand, Merton and others. 

 \defn{\lbl{weak_transport_defn}
Let $c: X\times {\mathcal P}(Y)\to  \R \cup\{+\infty\}$ be a bounded below, lower semi-continuous function such that for each $x\in X$, the function  $\sigma \mapsto c(x, \sigma)$ is proper and convex. The {\em optimal weak transport} problem with cost $c$ from $\mu \in \mcal{P}(X)$ to $\nu \in \mcal{P}(Y)$ is
\eq{\lbl{weak_transport1}
\V_c(\mu,\nu) := \inf_\pi\{\int_X c(x, \pi_x)\, d\mu(x); \pi \in {\mathcal K}(\mu, \nu)\}
}
where $(\pi_x)_x$ is the disintegration of $\pi$ with respect to $\mu$. 
}
\noindent As shown in \cite{ABC} and 
\cite{Go}, 
$\V_c(\mu,\nu) =\sup\{\int_Yg\, d\nu - \int_XT^-g\, d\mu \},$  
where 
\begin{equation}
T^- g(x) := \sup\{ \int_{Y}gd\sigma - c(x,\sigma)\,;\, \sigma \in \mcal{P}(Y)\}, 
\end{equation}
a duality that will be key to our study.
The effective domain of $\V_c$
is not necessarily a backward transfer set, but  we shall show in Section 3 that the domain of every backward linear transfer is contained in a minimal backward transfer set. \\
 
 \noindent{\it Optimal balayage transport with cost:} This example which combines balayage  and optimal transport can be described as follows: 
 
  \defn{\lbl{balayagetransport_defn} Assume $\A$ is a balayage cone in $LSC(X \sqcup Y)$,
  and let $c: X\times {\mathcal P}(Y)\to  \R \cup\{+\infty\}$ be as in definition (\ref{weak_transport_defn}).
  The following is also a backward linear transfer:
\begin{equation}\label{optimalbalayage}
{\mathcal B}_{c,\A}(\mu,\nu) = \begin{cases}
\inf\{\int_{X} c(x, \pi_x)d\mu(x)\,;\, \pi \in \mathcal{K}_{\A}(\mu,\nu)\} & \text{if }\mu \preceq_{\A} \nu,\\
+\infty & \text{otherwise.}
\end{cases}
\end{equation}
The corresponding Kantorovich operator is then
 \eq{
 T^-f(x)=\sup\{\int_Yf\d\sigma -c(x, \sigma);\, (x, \sigma)\in  X\times {\mathcal P}(Y)\, \hbox{and $\delta_x\prec_\A \sigma$}\}.
 }
We shall refer to it as an {\bf optimal restricted balayage transport} with cost $c$. 

If $X=Y$ and $\A$ is a balayage cone in $LSC(X)$, then ${\mathcal B}_{c,\A}$ is said to be an {\bf optimal balayage transport} with cost $c$. 
}

The following summarizes our main results. It roughly shows that mass transport theory is essentially balayage theory (or potential theory) with a ``cost" assigned for moving distributions around. 

 \thm{\label{prop.main1} Let $\T$ be a standard functional on $\P(X)\times {\mathcal P}(Y)$. Then the following are equivalent:
 \begin{enumerate}
 \item $\T$ is a backward linear transfer.
 
 \item $\T$ is an optimal restricted balayage transport, i.e., $\T={\mathcal B}_{c,\A}$  with cost function $c(x,\sigma) := \T(\delta_x, \sigma)$ and where  $\A$ is a balayage cone on $X \sqcup Y$.  
 
   Moreover, for any 
$(\mu, \nu)\in D(\T)$, there exists 
 $\pi\in  {\mathcal K}_\A (\mu, \nu)$ such that $(\delta_x, \pi_x)\in D(\T)$, $\delta_x\prec_\A \pi_x$ for $\mu$-a.e.$x\in X$ and 
$
\T(\mu, \nu)=\int_X \T(\delta_x, \pi_x)\, d\mu(x). \label{superStrassen1}
$
   \end{enumerate}
   Furthermore, 
   \begin{itemize}
 \item  $D(\T)= D({\mcal B}_{c, \A})\subset \{(\mu, \nu) \in \P(X)\times \P(Y);\, \mu \prec_{\mcal A} \nu\}$, and we have equality if and only if $D(\T)$ is a backward transfer set, in which case,
 \[
 D(\T)=\{(\mu, \nu)  \in \P(X)\times \P(Y); \hbox{$\exists$ 
 $\pi\in  {\mathcal K}_\A (\mu, \nu)$, $\delta_x\prec_\A \pi_x$ for $\mu$-a.e. $x\in X$}\}.
 \]
 \item If $X=Y$ and $T^-$ is the backward Kantorovich operator associated with $\T$, then  $D(\T)
 \subset \widehat {D({\mcal B}_{c, \hat \A})}:=\{(\mu, \nu) \in \P(X)\times \P(X);\, \mu \prec_{\hat {\mcal A}} \nu\}$, where $\hat {\A}$ is the  cone of lower semi-continuous $T^+$-superharmonic functions on $X$.

 \item $D(\T)=\widehat {D({\mcal B}_{c, \hat \A})}$ if and only if $D(\T)$ is a transitive backward transfer set containing the diagonal $\{(\delta_x, \delta_x); x\in X\}$.
\end{itemize}
 }
\noindent Note that (\ref{superStrassen1}) can be seen as an extension of Strassen's theorem since the latter correspond to when $\T$ is the zero-cost balayage transfer.

There are two important connections that emerge from our analysis:
\begin{enumerate}

\item The class $\K(Y, X)$ of backward Kantorovich operators in the smallest ``manageable" convex set in the space ${\mathcal F}(Y, X)$ of maps from $C(Y)$ to $USC(X)$ that is stable under pointwise suprema, while containing the Markov operators. 

\item Essentially any map $T: C(Y)\to USC(X)$ (resp., any functional capacity) has a lower ``envelope" that is a backward Kantorovich operator. Moreover,  $\K(Y, X)$ is an important subclass of functional capacities that could be called ``convex functional Choquet capacities," since essentially any functional capacity has a backward Kantorovich operator lower envelope for the standard order on capacities. 
\end{enumerate}

These connections make the cone of Kantorovich operators particularly rich and flexible as it echoes the role of convex functions vis-a-vis general numerical functions.
 
 
  
 
 Here is an outline of the paper. In Section 2, we exhibit the duality between linear transfers and Kantorovich operators as well as some of their properties that will be needed later.  We also show that backward Kantorovich operators are Choquet functional capacities.  In Section 3, we prove Theorem \ref{transfer.set} and show that if $\T$ is a standard backward linear transfer, then its effective domain $D(\T)$ is contained in a minimal restricted balayage set. In Section 4,  we prove Theorem \ref{true.balayage0} and show that if $X=Y$, then $D(\T)$  is contained in a minimal true balayage set. In Section 5, we show that if $\T$ is a standard backward linear transfer,  then it can be represented as an \textit{optimal weak transport}, 
and characterize those whose domains $D(\T)$ are backward transfer sets.
This is then used in Section 6 to construct backward Kantorovich operators (resp., Choquet-Kantorovich operators) as envelopes of general maps from $C(Y)$ to $USC(X)$(resp., functional capacities).

  \section{Duality and Kantorovich operators as Choquet functional capacities}\label{backwardlineartransfers}
  
 We shall restrict our study throughout to the case where $X$ and $Y$ are two compact metric spaces. Extensions to the non-compact case are of course possible and sometimes desirable, but could add complications and technical hypotheses that would take a way from this natural setting.  Recall that $USC(X)$ is the cone of proper,  bounded above, upper semi-continuous functions on $X$. 
  
\defn{\lbl{back_Kant}A  \textbf{backward Kantorovich operator} is a map $T^-: C(Y) \to USC(X)$ that satisfies the following properties:
 
 \begin{enumerate}
 \item $T^-$ is {\it monotone}, i.e., if $g_1\leq g_2$ in $C(Y)$, then $T^-g_1\leq T^-g_2$. 
 \item $T^-$ is {\em affine on the constants}, i.e., for any $c\in \R$ and $g\in C(Y)$,   $T^-(g+c)=T^-g +c.$ 
 \item  $T^-$ is {\it a convex operator}, that is for any $\lambda \in [0, 1]$, $g_1, g_2$ in $C(Y)$, we have 
\begin{equation*}
T^-(\lambda g_1+(1-\lambda)g_2)\leq \lambda T^-g_1+(1-\lambda)T^-g_2. 
\end{equation*}
\item $T^-$ is \textit{lower semi-continuous} in the sense that if $g_n \to g$ in $C(Y)$ for the sup norm, then $\liminf_{n \to \infty}T^-g_n \geq T^-g$.
 \end{enumerate}
} 
Similarly, a  \textbf{forward Kantorovich operator} is is a map $T^+: C(X) \to LSC(Y)$ that satisfies (1)and (2) above, but instead of (3) and (4), it satisfies:

(3')  $T^+$ is {\it a concave operator}.

(4') $T^+$ is \textit{upper semi-continuous}.

\noindent This class of operators is in a way ``dual" to the class of backward linear transfers.

  \begin{theorem}\label{kanto} Let $X$ and $Y$ be two compact metric spaces. Then, the following assertions are equivalent:
  \begin{enumerate}
  \item $T^-$ is a backward Kantorovich operator from $C(Y)$ to $USC(X)$. 

\item There exists a backward linear transfer $\T :\P(X) \times \P(Y) \to \R\cup \{+\infty\}$ such that for all $\mu \in \P(X)$ and $g \in C(Y)$,
\eq{
\T_\mu^*(g) = \int_{X}T^- g\,  d\mu. 
}
\item There exists a proper lower semi-continuous function $c: X\times {\mathcal P}(Y)\to  \R \cup\{+\infty\}$ such that for all $x\in X$, the functional $\sigma \mapsto c(x, \sigma)$ is convex, and for any $g\in C(Y)$,
\eq{\lbl{4prop}
T^-g(x) := \sup\{ \int_{Y}gd\sigma - c(x,\sigma)\,;\, \sigma \in \P(Y) \}.  
}
\end{enumerate}
\end{theorem}
We shall need the following lemma.

\begin{lemma} Let $T: C(Y) \to USC(X)$ be a backward Kantorovich operator, and suppose $x\in X$ is such that 
$T0(x)>-\infty$.Then, there exists $\sigma \in \P(Y)$ such that 
\eq{\lbl{f1}
\sup_{g\in C(Y)}\{\int_Y g d\sigma -Tg(x)\}<+\infty.
}
\end{lemma}
\noindent{\bf Proof:} Note first that the hypothesis yields that $Tg(x)>-\infty$ for every $g\in C(Y)$, since by the first 2 properties of a Kantorovich operator, we have for every $g\in C(Y)$, $T g-\inf_Yg=T(g-\inf_Yg)\geq T0$.  It follows that the functional $\Phi_x: C(Y)\to \R$ defined by $\Phi_x(g)=Tg(x)$ is clearly convex, lower semi-continuous and finite. Hence, by the Hahn-Banach theorem, it is the supremum of all affine functionals below it, that is, for every $g\in C(Y)$, 
\[
\Phi_x(g)=\sup\{\int_Y g d\sigma + k; \sigma \in {\cal M}(Y), k\in \R, \sigma +k \leq \Phi_x\, {\rm on}\, C(Y)\}.
\]
By compactness, there is $k_0\in \R$ and $\sigma_0 \in {\cal M}(Y)$ such that  $\sigma_0 +k_0 \leq \Phi_x$ and 
\[
\Phi_x(0)=k_0=\sup\{k; \sigma \in {\cal M}(Y), k\in \R, \sigma +k \leq \Phi_x\, {\rm on}\, C(Y)\}.
\]
Since $\Phi_x$ is affine on constants, we have
$\sigma_0 (Y)+k_0 \leq \Phi_x(1)=1+\Phi_x(0)=1+k_0$, hence $\sigma_0(Y)\leq 1$. On the other hand, 
\[
\Phi_x(-1)+1=\Phi_x(0)=\int_Y(-1)\d\sigma_0 +k_0+\sigma_0(Y)\leq \Phi_x(-1)+\sigma_0(Y), 
\]
hence $\sigma_0(Y)\geq 1$ and therefore it is equal to $1$. 

To show that $\sigma_0$ is a non-negative measure, let $f\in C(Y)$ with $0\leq f\leq 1$, then 
\[
1-\sigma_0(f) =\sigma_0(1)-\sigma_0(f)=\sigma_0 (1-f)\leq \Phi_x(1-f)-k_0.
\]
Using the monotonicity of $\Phi_x$, we obtain 
$$\sigma_0 (f) \geq 1- \Phi_x(1-f)+k_0=-\Phi_x(-f)+k_0\geq -\Phi_x(0)+k_0=0.$$
It follows that $\sigma_0\in \P(Y)$, $\sigma_0+k_0 \leq \Phi_x$ on $C(Y)$ and therefore 
for every $g\in C(Y)$, we have 
\[
Tg (x)=\sup\{\int_Y g d\sigma + k; \sigma \in {\cal M}(Y), k\in \R, \sigma +k \leq \Phi_x\, {\rm on}\, C(Y)\}\geq \int_Y g d\sigma_0 + k_0. 
\]
\noindent{\bf Proof of Theorem \ref{kanto}:} $(1) \Rightarrow (2)$ since if $T^-$ is a backward Kantorovich operator then $T^-0\in USC(X)$, that is (\ref{f1}) holds for some $x \in X$, and the following functional  
\begin{equation}\label{duality.bis}
\T(\mu,\nu) = \begin{cases}
\sup_{g \in C(Y)}\{\int_{Y}gd\nu - \int_{X}T^-gd\mu\}\quad&\hbox{for all $(\mu, \nu)\in \P(X)\times \P(Y)$}\\
+\infty & \text{\rm otherwise,}
\end{cases}
\end{equation}
is therefore proper. $\T$ is also bounded below since 
$$\inf \{\T(\mu, \nu); (\mu, \nu) \in \P(X)\times \P(Y)\}\geq \inf\{-\int_XT^-0\, d\mu; \mu\in \P(X)\}> -\inf_X T(0)>-\infty,$$ because $T^-0$ is upper semi-continuous. It is clear that $\T$ is jointly convex and lower semi-continuous on ${\mathcal M}(X)\times {\mathcal M}(Y)$ since $T^-$ is valued in $USC(X)$.  Now define for $\mu \in  \P(X)$, the functional $F_\mu(g) := \int_{X}T^- gd\mu$.
 
 Assuming $F_\mu (g)\neq -\infty$ for some $g\in C(Y)$ and since $g \mapsto T^-g$ is convex and lower semi-continuous in the sense defined above,  $F_\mu$ is then a  convex and lower semi-continuous on $C(Y)$. By Fenchel-Moreau duality, we will be done if we show that 
$F^*_\mu(\nu) = \T(\mu,\nu)$ for all $\nu \in \mcal{M}(Y)$, where 
\eqs{
\hbox{$F^*_\mu(\nu) = \sup\{\int_{Y}gd\nu - \int_{X}T^- gd\mu\,;\, g \in C(Y)\}$ on $\mcal{M}(Y)$}.
}
The equality holds when $\nu \in \mcal{P}(Y)$ by the definition of $\T$, so suppose $\nu \in \mcal{M}(Y)$ with $\nu(Y) = \lambda \neq 1$. Taking $g(x) \equiv n \in \mathbb{Z}$, we have $T^-(g) = T^-(0 + n) = n + T^-(0)$, and therefore
\eqs{
F^*_\mu(\nu) \geq n\lambda - \int_{X}T^-(n) d\mu = n(\lambda-1) - \int_{X}T^-(0)d\mu.
}
With $n \to \pm \infty$, depending on if $\lambda < 1$ or $\lambda > 1$, we deduce $F^*_\mu(\nu) = +\infty$. Hence $F^*_\mu(\nu) = \T(\mu,\nu)$ for all $\nu \in \mcal{M}(Y)$, and it follows that 
\eqs{
F_\mu(g) = \sup_{\nu \in \mcal{M}(Y)}\{\int_{Y}gd\nu - F_\mu^*(\nu)\} = \sup_{\nu \in \mcal{M}(Y)}\{\int_{Y}gd\nu - \T(\mu,\nu)\} =\T_{\mu}^*(g).
}
If now $F_\mu\equiv -\infty$, that is if $F_\mu(g)=\int_{X}T^- gd\mu=-\infty$ for every $g\in C(Y)$, then formula (\ref{duality.bis}) gives that $\T_\mu\equiv +\infty$, and again $\T_\mu^*(g)=-\infty=\int_{X}T^- gd\mu.$
 
$(2) \Rightarrow (1)$:  Suppose now $\T$ is a backward linear transfer, and let $T^-$ be the operator associated to it in Definition \ref{lineartransfers}. Since $\T$ is proper, we have $\T(\mu, \nu)<+\infty$ for some $(\mu, \nu)\in \P(X)\times \P(Y)$, hence 
$$\int_{X}T^- 0\, d\mu=\T_\mu^*(0)=\sup\{ -\T(\mu, \sigma); \sigma \in \P(Y)\} \geq  -\T(\mu, \nu)>-\infty, 
$$ 
from which follows that there is $x\in X$ such that $T^-0(x)>-\infty$.  

To show that $T^-$ is a Kantorovich operator,  note that for every $g \in C(Y)$ and any $x\in X$,
\eq{\lbl{pointwise_express}
T^- g(x)=\T_{\delta_x}^*(f)=\sup_{\nu \in \mcal{P}(Y)}\{\int_{Y}gd\nu - \T(\delta_x, \nu)\}.  
}
It is clear that $T^-$ is monotone, convex and affine on the constants. To show that $T^- g \in USC (X)$ for every $g\in C(Y)$, note first that $T^-g$ is proper since $T^-0(x)>-\infty$ and the monotonicity yields that it is also the case for any $g\in C(Y)$. 
Moreover, if $x_n \to x$ in $X$ so that -up to a subsequence- $\limsup_{n \to \infty}T^-g(x_n) = \lim_{j \to \infty}T^-g(x_{n_j}),$ then 
there is nothing to prove if the latter is equal to $-\infty$. However, if $\lim_{j \to \infty}T^-g(x_{n_j})>-\infty$, then we 
let $\nu_j$ achieve the supremum above when $x = x_{n_j}$ (the supremum is achieved by upper semi-continuity of $\nu \mapsto \int_{Y}gd\nu - \T(\delta_x, \nu)$ on the compact space $\mcal{P}(Y)$).  By the weak$^*$ compactness of $\mcal{P}(Y)$, we may extract a further subsequence if necessary and assume $\nu_{j} \to \bar{\nu}$ for some $\bar{\nu} \in \mcal{P}(Y)$. It follows from the weak$^*$- lower semi-continuity of $\T$
\as{
\limsup_{n \to \infty}T^-g(x_n) = \lim_{j \to \infty}T^-g(x_{n_j}) \leq \int_{Y}gd\bar{\nu} - \T(\delta_x, \bar{\nu})
\leq \sup_{\nu \in \mcal{P}(Y)}\{\int_{Y}gd\nu - \T(\delta_x, \nu)\} = T^-g(x),
}
hence $T^- g \in USC(X)$.

This also yields the following Lipschitz property: for any $g, h$ in $C(Y)$, and any $x\in X$, where $T^-h(x)>-\infty$, 
\eq{\lbl{lip}
T^-g(x)-T^-h(x)\leq  \|g - h\|_\infty,
}
since if $T^-g(x)>-\infty$ (otherwise there is nothing to prove), then 
\eqs{
T^-g(x) = \sup_{ \sigma \in \mcal{P}(Y)}\{\int_{Y}gd\sigma - \T(\delta_x,\sigma)\}
\leq \sup_{ \sigma \in \mcal{P}(Y)}\{\int_{Y}hd\sigma - \T(\delta_x,\sigma)\} + \|g - h\|_\infty\\
= T^-h(x) + \|g - h\|_\infty.
}
This immediately yields the lower semi-continuity property since if $g_n \to g \in C(Y)$, then,
$T^-g(x) \leq T^-g_n(x) + \|g - g_n\|_\infty,$
and 
$T^-g \leq \liminf_{n \to \infty}T^-g_n(x).$

$(2) \Rightarrow (3)$ is immediate by taking $c(x, \sigma)=\T(\delta_x, \sigma)$ and using the definition of a Legendre transform. 

That $(3) \Rightarrow (1)$ is immediate from expression (\ref{4prop}), which was used in $(2) \Rightarrow (1)$ to establish the 4 properties of a Kantorovich operator.  

Finally, note that for each $\mu\in \P(X)$, $\T_\mu$ is convex, lower semi-continuous and bounded below such that $\T_\mu^*(g)=\int_XT^-g\, d\mu$. It follows that for each $\nu\in \P(Y)$,
\[
\T(\mu, \nu)=\T_\mu(\nu)=\T_\mu^{**}(\nu)=\sup_{g \in C(Y)}\{\int_{Y}g\, d\nu - \T_\mu^*(g)\}= \sup_{g \in C(Y)}\{\int_{Y}gd\nu - \int_{X}T^-gd\mu\}.
\]
\begin{proposition} Consider the class ${\mathcal K}(Y, X)$ of backward Kantorovich operators from $C(Y)$ to $USC(X)$. Then, 
\begin{enumerate}

\item If $T_1$ and $T_2$ are in ${\mathcal K}(Y, X)$, and $\lambda \in [0, 1]$, then $\lambda T_1+(1-\lambda)T_2 \in {\mathcal K}(Y, X)$. 

\item If $\lambda \in \R^+$ and $T\in {\mathcal K}(Y, X)$, then the map $(\lambda \cdot T)f:=\frac{1}{\lambda}T(\lambda f)$ for any $f\in C(Y)$, belongs to ${\mathcal K}(Y, X)$.


\item If $T_1$ and $T_2$ are in ${\mathcal K}(Y, X)$, then the map $T_1\star T_2$ defined for $f\in C(Y)$ and $x\in X$ by 
\[
(T_1\star T_2)f(x):=\sup_{\sigma \in {\mathcal P}(Y)}\inf_{g, h\in C(Y)}\int_Y(f-g-h)\, d \sigma +T_1g(x) +T_2h(x) 
\]
is in ${\mathcal K}(Y, X)$

\item  If $T_1$ and $T_2$ are in ${\mathcal K}(Y, X)$, then the map $(T_1\vee T_2)f(x):=T_1f(x)\vee T_2f(x)$ for any $f\in C(Y)$, and $x\in X$ belongs to ${\mathcal K}(Y, X)$.  

\item If $(T_i)_{i\in I}$ is a family  in ${\mathcal K}(Y, X)$, such that for every $f\in C(Y)$, the function $\sup_{i\in I}T_if$ is bounded above on $X$, then the operator $
f\to \overline{\sup_{i\in I}T_if}$ 
is a Kantorovich operator, where here $\overline g (x)=\inf\{h(x); h\in C(X), h\geq g \, {\rm on}\, X\}$ is the smallest upper semi-continuous function above $g$. 
\end{enumerate}
\end{proposition}
\prf{ (1) and (2) are immediate, while (3) follows from the fact that 
\[
T_1\star T_2 f(x)=\sup_{\sigma \in {\mathcal P}(Y)}\{ \int_Y f \, d\sigma - \T_1(x, \sigma)-\T_2(x, \sigma\},
\] 
where $\T_1$ and $\T_2$ are the linear transfers corresponding to $T_1$ and $T_2$ respectively. In other words, in view of Theorem \ref{kanto}, the convex cost associated with $T_1\star T_2$ is 
\[
c(x, \sigma)=\T_1(\delta_x, \sigma)+\T_2(\delta_x, \sigma), 
\]
hence $T_1\star T_2$ is a Kantorovich operator. 

4) is immediate from the definition of a Kantorovich operator, while for 5) first notice that the operator $T_\infty f:=\sup_{i\in I}T_if$
clearly satisfies properties $(1),...,(4)$ of a backward Kantorovich operator. Indeed, the monotonicity, convexity, and affine with respect to constants, properties of $S_\infty$ are readily inherited from each $S_i$ being Kantorovich. The map $f\to T_\infty f(x)$ is also lower semi-continuous. The first three properties extend to the operator $f\to \overline{T_\infty f}$. It remains to check the lower semi-continuity.  For this, suppose $g_k \to g$ in $C(X)$ for the sup norm. If $\liminf_k T_\infty g_k(x) =+\infty$, there is nothing to prove. Otherwise, 
we have from (\ref{lip}) that 
$T_ig(x) \leq T_ig_k(x) + \|g-g_k\|_\infty$, 
so that 
$\sup_iT_ig(x) \leq \sup_iT_i g_k(x) + \|g-g_k\|_\infty$ and therefore 
$$\overline{\sup_iT g}(x) \leq \overline{\sup_iT_ig_k}(x) + \|g-g_k\|_\infty,$$
 since $\overline{f+c}=\overline{f} +c$ for any constant $c$, hence $\overline{T_\infty g}(x) \leq \liminf_{k \to \infty} \overline {T_\infty g_k}(x)$. 
}
\begin{remark} The duality allows to construct and identify new Kantorovich operators. Indeed, let $I: \mcal{P}(Y) \to \R\cup\{+\infty\}$ be a bounded below, convex, weak$^*$-lower semi-continuous function on $\mcal{P}(Y)$ and consider 
\begin{equation}\lbl{convex_energy}
{\cal I}(\mu, \nu) := I(\nu)\quad  \text{for all  $(\mu, \nu) \in \mcal{P}(X) \times \mcal{P}(Y),$}
\end{equation}
so that ${\cal I}$ is a backward linear transfer with corresponding backward Kantorovich operator $T^-g \equiv I^*(g)$ (hence a constant function of $x$).  If now $\T$ is a general linear transfer, then the following functional 
on $\P(X)\times \P (Y)$, 
\eq{
\T_{I}(\mu, \nu):
=\inf\{\int_{X}[\T(x, \pi_x)\ + I (\pi_x)]\, d\mu(x); \pi \in {\mathcal K}(\mu, \nu)\}
}
is also a linear transfer, and 
if $\T$ is given by an optimal mass transport $\T_c$ associated to a proper lower semi-continuous cost function $c:X\times Y \to \R\cup\{+\infty\}$, whose backward Kantorovich operator is $T_cg(x):=\sup\{g(y)-c(x, y); y\in Y\}$, then  
\eq{
\T_{I}(\mu, \nu)=\inf\{\int_{X} [c(x, y)\, d\pi_x(y)  + I (\pi_x)]\, d\mu(x); \pi \in {\mathcal K}(\mu, \nu)\},
}
and the associated Kantorovich operator is given by 
 \as{
T_I^- g(x):=T^-\oplus I^*(g)(x) &= \sup_{\sigma \in \mcal{P}(Y)}\{\int_{Y}g d\sigma - \T(x,\sigma)-I(\sigma)\}\\
   &= \sup\{\int_{Y}(g(y)-c(x,y))d\sigma(y)-I(\sigma)\,;\, \sigma \in \mcal{P}(Y)\}\\
&=I^*({g}_{c, x}),
}
where ${g}_{c, x}$ is the function $y\mapsto g(y)-c(x, y)$.

\begin{enumerate}
\item If $I$ is the potential energy functional $I(\nu)=\epsilon \int_YV(y)\, d\nu (y)$, where $V$ is a bounded below lower semi-continuous potential on $Y$, then $T_I^- g= T^-_c(g-\epsilon V)$.

 \item Less obvious is the case when $I$ is the relative entropy $H_{\nu_0}$ with respect to $\nu_0$, i.e.,  
\eq{
H_{\nu_0}(\nu) := \begin{cases}
\int_Y \frac{d\nu}{d\nu_0}\log (\frac{d\nu}{d\nu_0}) d\nu_0 & \text {if $\nu <<\nu_0$}\\ 
+\infty & \text{otherwise.}
\end{cases}
}
In this case, $T_{\epsilon H}f(x)=\epsilon \log \int_{Y}e^{\frac{g(y)-c(x, y)}{\epsilon}}d\nu_0(y)$, which corresponds to an entropic regularization of $T_c$. 

\end{enumerate}

\end{remark}

Denote by $USC_f(X)$ (resp., $USC_b(X)$) the cone of functions in $USC(X)$ that are finite (resp., bounded below), and by $USC_\sigma(X)$ the closure of $USC(X)$ with respect to monotone increasing limits. 

\begin{definition} Say that a Kantorovich operator $T$ is {\bf standard} (resp., {\bf regular}) if  $T$ maps $C(Y)$ to $USC_f(X)$ (resp., $USC_b(X)$).
\end{definition}
The following proposition follows readily from the above theorem. 
 
\begin{proposition}  Let $T$ be a Kantorovich operator and $\T$ its associated linear transfer. Then, the following are equivalent:
\begin{enumerate}
\item $T$ is {\bf standard} (resp., {\bf regular}).

\item The function $T(0)$ is finite (resp., bounded below).

\item The corresponding linear transfer $\T$ satisfies 
\eq{
\hbox{$\inf\limits_{\sigma \in \P(Y)}\T(x, \sigma)<+\infty$ for all $x\in X$,
} 
}

\eq{
\hbox{(resp.,}\quad  k:=\sup_{x\in X}\inf_{\sigma \in \P(Y)}\T(x, \sigma)<+\infty. )
}
\end{enumerate}
If $T$ is a standard Kantorovich operator, then $T$ is a contraction from $C(Y)$ to $USC_f(X)$, that is 
\eq{\lbl{lip1}
\sup_{x\in X}|T^-g(x)-T^-h(x)|\leq  \|g - h\|_\infty.
}
A regular operator is clearly standard and $T+k$ then maps $C_+(Y)$ to $USC_+(X)$ for each $x\in X$. 
\end{proposition}
 
\defn{{\bf A Choquet functional capacity} is a map $T: F_+(Y)\to F_+(X)$ such that
\begin{enumerate}
\item $T$ is monotone, i.e., $f\leq g \Rightarrow Tf\leq Tg$. 

\item $T$ maps $USC(Y)$ to $USC(X)$ and if $g_n, g \in USC(Y)$ and $g_n \downarrow g$, then $T^- g_n \downarrow T^-g$. 

\item If $g_n, g \in F_+(Y)$  with $g_n \uparrow g$, then $T^- g_n \uparrow T^- g$.\\
\end{enumerate}
}

 The class $USC_\sigma(X)$ will denote the closure of $USC(X)$ with respect to monotone increasing limits.

 \begin{theorem}\lbl{capacity}Let $\T$ be a backward linear transfer  and let $T^-:C(Y) \to USC(X)$ be the associated  Kantorovich operator. 
 Then
 \begin{enumerate}
 \item $T^-$ can be extended to be a functional from $F(Y)$ to $F(X)$ via the formula  \begin{equation}\label{express.1}
T^- g (x)=\sup\{\int^*_Y g d\nu -{\mathcal T}(\delta_x, \nu); \, \nu \in \P(Y), (\delta_x, \nu) \in D(\T)\},
\end{equation}
where $\int^*_Y g d\nu$ is the outer integral of $g$ with respect to $\nu$.\\
Moreover,  $T^-$ maps bounded above  functions on $Y$ to  bounded above functions on $X$. It is monotone and satisfies $T^-g+c=T^-(g+c)$ for every $g\in F(Y)$ and $c\in \R$. 

  \item For any upper semi-continuous functions $g$ on $Y$,  we have
 \eq{\lbl{rep1}
 T^{-} g (x) := \inf\{ T^- h(x)\,;\, h \in C(Y), \, h \geq g \}, 
 }
 and  $T^-$ maps $USC_f (Y)$ 
  to $USC(X)$. 
   \item If $T^-$ is standard (resp., regular), then it maps $USC_f (Y)$ (resp., $USC_b(Y)$) to $USC_f (X)$ (resp., $USC_b(Y)$). 
  \item If $T^-$ is regular, then the map $T^-+k$ is a functional capacity that maps $F_+(Y)$ to $F_+(X)$, and consequently, if $g$ is a $K$-analytic function that is bounded on $Y$, then
 \eq{\lbl{rep2}
T^{-}g(x) := \sup\{ T^{-} h(x)\,;\, h \in USC(Y)\,,  h \leq g\}.
}
\end{enumerate} 

\end{theorem}
\prf{1) Note first that formula (\ref{express.1}) makes sense for any  function $g:Y\to  \R_+\cup\{+\infty\}$
provided one uses the outer integral $\int^*g\, d\nu$ in formula (\ref{express.1}).
Moreover, $T^{-}g(x) \leq \sup_{y\in Y} g(y) - m_{\T}$, where $m_{\T}$ is a lower bound for $\T$, hence if $g$ is bounded above then, $T^{-}g$ is also bounded above. The monotonicity and the affine properties on constants are clear. 
Finally note that since there exists $(\delta_x, \sigma) \in D(\T)$, we have for any $g\in F_b(Y)$, $T^-g(x)\geq \inf_Yg-\T (\delta_x, \sigma)>-\infty$, hence  $T^-g$ is proper.

2. We now show that the extension  as defined in (\ref{express.1}) and still denoted $T^-$,  maps $USC(Y)$ to $USC(X)$. 
Indeed, if $g \in USC(Y)$ and $h_n \searrow g$ is a decreasing sequence of continuous functions converging to $g$, then, $T^{-} g \leq \inf_nT^- h_n$ so that if $\inf_nT^- h_n(x)=-\infty$, we have equality at that point. On the other hand, if  $\inf_nT^- h_n(x)>-\infty$, then
\as{
T^{-} g(x) \leq \inf_nT^- h_n(x) =\inf_n \sup_{\sigma \in \mcal{P}(Y)}\{\int_Y h_nd\sigma - \T(\delta_x, \sigma)\}
=\inf_n \int_Y h_n d \sigma_n - \T(\delta_x, \sigma_n),}
where the supremum is achieved for some probability measure $\sigma_n$ because $\sigma \mapsto \int_Y h_nd\sigma - \T(\delta_x, \sigma)$ is weak$^*$ upper semi-continuous and bounded above on the compact set $\mcal{P}(Y)$.

 Let $(n_k)_k$ be a subsequence such that $\inf_n \int_Y h_n d \sigma_n - \T(\delta_x, \sigma_n)=\lim_{k \to \infty}\int_Y h_{n_k} d \sigma_{n_k} - \T(\delta_x, \sigma_{n_k})$ and by weak$^*$ compactness of $\mcal{P}(Y)$, extract a further  increasing subsequence (that we call again $n_k$) so that $\sigma_{n_k} \to \overline{\sigma}$. For any $j \leq k$,
we have $h_{n_k} \leq h_{n_j}$, hence for this fixed $j$, we have that $h_{n_j} \in C(Y)$ and so $\int h_{n_j}d\sigma_{n_k} \to \int h_{n_j}d\overline{\sigma}$ as $k \to \infty$, and since $\T$ is lower semi-continuous, we obtain
$$
\inf_n \int_Y h_n d \sigma_n - \T(\delta_x, \sigma_n)  \leq \lim_{k \to \infty}\int_Y h_{n_j}d\sigma_{n_k} - \T(\delta_x, \sigma_{n_k})\\
\leq \int_Y h_{n_j}d\overline{\sigma} - \T(\delta_x, \overline{\sigma}).
$$
Finally $\int_Y h_{n_j}d\overline{\sigma} \to \int_Y g d\overline{\sigma}$ by monotone convergence, and we obtain 
$$
T^{-} g(x) \leq \inf_nT^- h_n(x)\leq \inf_n \int_Y h_n d \sigma_n - \T(\delta_x, \sigma_n)  \leq \sup_{\sigma \in \mcal{P}(Y)}\{ \int_Y gd\sigma - \T(\delta_x, \sigma)\}=T^{-} g(x).
$$
It follows that 
$ T^{-} g (x) := \inf\{ T^- h(x)\,;\, h \in C(Y), \, h \geq g \}, $
 and $T^- g$ is upper semi-continuous. If now $g \in USC_f(Y)$, then $T^-g$ is proper and therefore belongs to $USC(X)$.  

3) Suppose now that $g_n, g \in USC(Y)$ are such that  $g_n \searrow g$. We claim that $T^- g_n \searrow T^-g$. Indeed, by the monotonicity property of $T^-$, we have $T^- g(x) \leq \liminf_{n} T^-g_n(x)$. If for some $n$, $T^- g_n(x) = -\infty$, then $T^-g(x) = -\infty$ and there is nothing to prove. Otherwise, $T^-g_n(x) > -\infty$ for all $n$, in which case the expression \refn{express.1} is finite. The map
$\sigma \mapsto \int_{Y}g_nd\sigma - \T(\delta_x, \sigma)$
is weak$^*$ upper semi-continuous (since $\sigma \mapsto \int_{Y}hd\sigma$ is weak$^*$ upper semi-continuous for any $h \in USC(Y)$), so it achieves its supremum at some $\sigma_n$, i.e.,  
$T^- g_n (x) = \int g_n d\sigma_n - \T(\delta_x, \sigma_n).
$\\
Extract now an increasing subsequence $n_k$ so that $\limsup_{n} T^- g_n (x) = \lim_{k}T^- g_{n_k}(x)$ and $\sigma_{n_k} \to \bar{\sigma}$. Similarly to the proof above, 
we get by the monotonicity of $g_{n}$, that
\eq{\lbl{portmanteau1}
T^- g_{n_k} (x) \leq \int g_{n_j}d\sigma_{n_k} - \T(\delta_x, \sigma_{n_k}) \quad \text{for fixed $j \leq k$.}
} 
As $g_{n_j} \in USC(Y)$ and $\sigma_{n_k} \to \bar{\sigma}$, it follows that $\limsup_{k \to \infty} \int g_{n_j}d\sigma_{n_k} \leq \int g_{n_j}d \bar{\sigma}$. Hence upon taking $\limsup_{k \to \infty}$ in \refn{portmanteau1}, we get that 
$
\limsup_{n} T^- g_n (x) \leq \int_Y g_{n_j}d \bar{\sigma} - \T(\delta_x, \bar{\sigma}).
$
Now let $j \to +\infty$ and use monotone convergence to conclude that
\[
\limsup_{n} T^- g_n (x) \leq \sup_{\sigma \in \mcal{P}(Y)}\{\int gd \sigma - \T(\delta_x, \sigma)\} = T^-g(x). 
 \]
Suppose now $g_n, g \in F(Y)$ with $g_n \nearrow g$, we shall show that $T^- g_n \nearrow T^- g$. Indeed, again, by the monotoncity property for $T^-$, we have $T^-g(x)  \geq \limsup_{n} T^- g_n(x)$. On the other hand,  
\eqs{
T^- g_n(x) = \sup_{\sigma \in \mcal{P}(Y)}\{\int^*_{Y}g_nd\sigma - \T(\delta_x, \sigma)\} \geq \int^*_Y g_n d\sigma - \T(\delta_x, \sigma) \quad \text{for all $\sigma \in \P(Y)$,}
}
hence by monotone convergence, $\liminf_{n} T^- g_n (x) \geq \int^* g d \sigma - \T(\delta_x, \sigma)$ for all $\sigma$. Taking the supremum over $\sigma$ yields $\liminf_{n} T^- g_n (x) \geq  \sup_{\sigma \in \mcal{P}(Y)}\{\int^*_{Y}gd\sigma - \T(\delta_x, \sigma)\} = T^- g(x)$.

4) If now $T$ is a regular Kantorovich operator, then $T$ maps $USC_b(Y)$ to $USC_b(X)$ and $T+k$ maps $F_+(Y)$ to $F_+(X)$. It follows from the above established properties that 
 $T+k$ is a functional capacity. Note that for each $x\in X$,  the set function $T^x(A)=T(\chi_A)(x)+k$ is then a non-negative regular Choquet capacity. Moreover, for any Radon measure $\mu$, the functional $(T+k)_\#\mu$ defined as 
$(T+k)_\#\mu(f)=\int_X(Tf (x)+k)\d\mu(x)$ is also a regular capacity. 

A celebrated theorem of Choquet now yields that any $K$-analytic function $g \in  F_+(Y)$, is ``capacitable", that is 
\[
T^{-}g(x) = \sup\{ T^{-} h(x)\,;\, h \in USC(Y)\,,  h \leq g\}. 
\]
This holds in particular for any bounded above function $g\in USC_\sigma(Y)$, i.e., the closure of $USC(Y)$ with respect to monotone increasing limits.
Actually, it can be readily be seen that
$T^{-}g(x) \geq \sup\{ T^{-} h(x)\,;\, h \in USC(Y)\,,  h \leq g\}$, and 
for an increasing $h_n \nearrow g$, with $h_n \in USC(Y)$, we have $T^-g=\sup_nT^-h_n$ by the 3rd property of the capacity. 
}
\coro{ \lbl{legendre_extension}Let  ${\T}: {\mathcal P}(X)\times {\mathcal P}(Y)\to \R \cup\{+\infty\}$ be a backward linear transfer. Then, 
\begin{enumerate} 

\item For any $(\mu, \nu)\in {\mathcal P}(X)\times {\mathcal P}(Y)$, we have
\begin{align*}
 {\mathcal T}(\mu, \nu) =\sup\big\{\int_{Y}g\, d\nu-\int_{X}{T ^-}g\, d\mu;\,  g \in USC_b(Y)\big\}.
\end{align*}
\item The Legendre transform formula \refn{Legendre_trans} for $\T_\mu$ extends from $C(Y)$ to $USC(Y)$; that is, for $\mu \in \mcal{P}(X)$,  we have for any $g\in USC(Y)$,
\begin{equation}\label{extLT}
{\mathcal T}^*_\mu (g):=\sup\{\int_Yg d\sigma -\T(\mu, \sigma); \sigma \in {\mathcal P}(Y)\}=
 \int_XT ^-g \, d\mu.
\end{equation}

\end{enumerate}
}
\prf{ 1. For $g \in USC_b(Y)$, take a monotone decreasing sequence $g_n \in C(Y)$ with $g_n \downarrow g$. Since $T^-$ is a capacity, we have
\eqs{
\lim_{n \to \infty}\lf(\int_{Y}g_nd\nu - \int_{X}T^- g_nd\mu\rt) = \int_{Y}gd\nu - \int_{X}T^- gd\mu,
}
from which we conclude
$\T(\mu,\nu) \geq \sup\lf\{\int_{Y}g\, d\nu-\int_{X}{T ^-}g\, d\mu;\,  g \in USC_b(Y)\rt\}.$
The reverse inequality is immediate since $C(Y) \subset USC_b(Y)$.

2. Let $g \in USC(Y)$ and take $g_n \searrow g$ with $g_n \in C(Y)$. We have
\eqs{
\int_Yg_n d\sigma -\T(\mu, \sigma) \leq \sup\{\int_Yg_n d\sigma -\T(\mu, \sigma); \sigma \in {\mathcal P}(Y)\}=\int_XT ^-g_n \, d\mu
}
so that by the capacity property of $T^-$, we have $\int_Yg d\sigma -\T(\mu, \sigma)\leq \int_XT ^-g \, d\mu$, 
and consequently
\eq{\lbl{Legendre_ext}
\sup\{\int_Yg d\sigma -\T(\mu, \sigma); \sigma \in {\mathcal P}(Y)\} \leq \int_XT ^-g \, d\mu.
}
On the other hand, by monotonicity, $T^- g \leq T^-g_n$, so 
\eq{\lbl{temp}
\int_XT ^-g \, d\mu \leq \int_XT ^-g_n \, d\mu \leq \sup\{\int_Yg_n d\sigma -\T(\mu, \sigma); \sigma \in {\mathcal P}(Y)\}.
}
If now $\sup\{\int_Yg_n d\sigma -\T(\mu, \sigma); \sigma \in {\mathcal P}(Y)\}=-\infty$ for some $n$, there is nothing left to prove. Otherwise, 
The supremum on the right-hand side of (\ref{temp}) is achieved by some $\sigma_n$. Extract an increasing subsequence $n_j$ so that $\sigma_{n_j} \to \sigma$ for some $\sigma \in \mcal{P}(Y)$. Then if $i \leq j$, we have $g_{n_j} \leq g_{n_i}$, so that
\as{
\int_XT ^-g \, d\mu \leq \int_Yg_{n_i} d\sigma_{n_j} -\T(\mu, \sigma_{n_j})\quad \text{for $i \leq j$}
}
where upon sending $j \to \infty$ yields 
$\int_XT ^-g \, d\mu \leq \int_Yg_{n_i} d\sigma -\T(\mu, \sigma)$, 
and finally as $i \to +\infty$,  monotone convergence yields the reverse inequality of \refn{Legendre_ext}.
}

Now that standard Kantorovich operators can be  extended so as to map $USC_f(Y)$ into $USC_f(X)$, we can compose them in the following way.

 \prop{\label{conv} Let $X_1,...., X_n$ be $n$ compact spaces, and suppose for each $i=1,..., n$,  ${\mathcal T}_i$ is a standard   backward linear transfer on ${\mathcal P}( X_{i-1})\times {\mathcal P}(X_i)$ with  Kantorovich operator  $T _i^-:C(X_{i})\to USC_f(X_{i-1})$.
 For any probability measures $\mu$ on $X_1$ (resp., $\nu$ on $X_n$), define
\[
\T_1\star\T_2...\star\T_n(\mu, \nu):=\inf\{{\mathcal T}_{1}(\mu, \sigma_1) + {\mathcal T}_{2}(\sigma_1, \sigma_2) ...+{\mathcal T}_{n}(\sigma_{n-1}, \nu);\, \sigma_i \in {\mathcal P}(X_i), i=1,..., n-1\}.
\]
Then,  ${\mathcal T}:=\T_1\star\T_2...\star\T_n$ is a linear backward transfer with a Kantorovich operator 
 \eq{
T ^-=T ^-_1\circ T ^-_2\circ...\circ T ^-_{n}.
 }
In other words, the following duality formula holds:
\begin{equation}
{\mathcal T}(\mu, \nu)=\sup\big\{\int_{X_n}g(y)\, d\nu(y)-\int_{X_1}T ^-_1\circ T ^-_2\circ...\circ T ^-_{n}g(x);\,  g\in C(X_n) \big\}.
\end{equation}
}
\prf{
By an obvious induction, it suffices to show the proposition for two transfers. For that, note that since  ${\mathcal T}_{1}$ (resp., ${\mathcal T}_{2}$) is jointly convex and weak$^*$-lower semi-continuous on ${\mathcal P}(X_1)\times {\mathcal P}(X_2)$ (resp., ${\mathcal P}(X_2)\times {\mathcal P}(X_3)$),  
 then 
 $({\mathcal T}_{1}\star{\mathcal T}_{2})_\mu: \nu \to ({\mathcal T}_{1}\star{\mathcal T}_{2})(\mu. \nu)$ is convex and weak$^*$-lower semi-continuous. Consider their corresponding Kantorovich operator $T^-_1$ (resp., $T_2^-$) from $USC(X_2)$ to $USC(X_1)$ (resp.,  $USC(X_3)$ to $USC(X_2)$ and calculate the following Legendre transform: For $g\in C(X_3)$,
 \begin{eqnarray*}
 ({\mathcal T}_{1}\star{\mathcal T}_{2})_\mu^*(g)&=&\sup\limits_{\nu \in {\mathcal P}(X_3)}\sup\limits_{\sigma \in  {\mathcal P}(X_2)} 
 \left\{\int_{X_3} g\, d\nu -{\mathcal T}_{1}(\mu, \sigma) - {\mathcal T}_{2}(\sigma, \nu)\right\}\\
&=&  
\sup\limits_{\sigma \in  {\mathcal P}(X_2)} 
 \left\{({\mathcal T}_{2})_\sigma^* (g)-{\mathcal T}_{1}(\mu, \sigma) \right\}\\
&=& \sup\limits_{\sigma \in  {\mathcal P}(X_2)} 
 \left\{\int_{X_2} T _2^-(g)\, d\sigma-{\mathcal T}_{1}(\mu, \sigma) \right\}\\
 &=&({\mathcal T}_{1})_\mu^* (T _2^-(g))\\
 &=& \int_{X_1} T _1^-\circ T _2^-g\, d\mu.
 \end{eqnarray*} 
  In other words, 
$
 {\mathcal T}_{1}\star{\mathcal T}_{2}(\mu, \nu)=
\sup\big\{\int_{X_3}g(x)\, d\nu(x)-\int_{X_1} T _1^-\circ T _2^-g\, d\mu;\,  f\in C(X_3) \big\},
$
which means that $\T_1\star \T_2$ is  a backward linear transfer on $X_1\times X_3$ with Kantorovich operator equal to $T _1^-\circ T _2^-$. 
}

\section{Gambling houses and positively homogenous Kantorovich operators}
In this section we study $\{0, +\infty\}$-valued linear transfers, i.e., those corresponding to a $0$-cost $c$ in expression $(\ref{4prop})$. 

 \begin{definition} Let ${\mcal S}$ be a non-empty subset of $\mcal{P}(X)\times\mcal{P}(Y)$.
\begin{enumerate}
\item 
Recall that ${\mcal S}$ is a {\bf gambling house} if for every $x\in X$, there exists $\sigma \in \P(Y)$ such that $(\delta_x, \sigma)\in \S$, that is, if its characteristic function 
 \begin{equation*}
\T_{\cal S}(\mu, \nu):=\left\{ \begin{array}{llll}
0 \quad &\hbox{if $(\mu, \nu)\in {\cal S}$}\\
+\infty \quad &\hbox{\rm otherwise,}
\end{array} \right.
\end{equation*} 
is a standard functional. 

\item Say that a subset ${\mcal S}$ of $\mcal{P}(X)\times\mcal{P}(Y)$ is a {\bf backward transfer set} if its characteristic function 
$\T_{\cal S}$
is a backward linear transfer. 
\end{enumerate}
\end{definition}
It is clear that transfer sets are convex weak$^*$-compact subsets of $\mcal{P}(X)\times\mcal{P}(Y)$. 

\begin{proposition} A standard backward Kantorovich operator $T$ is $1$-positively homogenous if and only there exists a closed  gambling house $\S$ such that 
\begin{equation}\lbl{homo1}
T^-g(x)=\sup\{\int_Xg\, d\sigma; (\delta_x, \sigma)\in \S\},
\end{equation}
in which case 
\begin{equation}\lbl{S}
\overline{\rm conv (\S)}=\{(\mu, \nu)\in \mcal{P}(X)\times\mcal{P}(Y); \, \nu \leq T_\#\mu \,\, {\rm on}\,\, C(Y)\},
\end{equation}
 where  $T_\#\mu (g):=\int_XTg \, d\mu$ for every $g\in C(Y)$. 
\end{proposition}

\prf{
Note that if $T^-$ is a $1$-positively homogenous Kantorovich operator, then the corresponding backward linear transfer 
\begin{equation}
\T(\mu,\nu) =
\sup_{g \in C(Y)}\{\int_{Y}gd\nu - \int_{X}T^-gd\mu\}\quad\hbox{for all $(\mu, \nu)\in \P(X)\times \P(Y)$}
\end{equation}
can only take values in $\{0, +\infty\}$, since then $$\T(\mu,\nu)\geq n\sup_{g \in C(Y)}\{\int_{Y}gd\nu - \int_{X}T^-gd\mu\},$$ for every $n$, which means it is $+\infty$ unless $\int_{Y}gd\nu - \int_{X}T^-gd\mu\leq 0$ for every $g\in C(Y)$. } In other words, 
 \begin{equation*}
\T (\mu, \nu)=\left\{ \begin{array}{llll}
0 \quad &\hbox{if $(\mu, \nu)\in {\cal D}$}\\
+\infty \quad &\hbox{\rm otherwise,}
\end{array} \right.
\end{equation*} 
where ${\cal D}:=\{(\mu, \nu)\in \mcal{P}(X)\times\mcal{P}(Y); \, \nu \leq T_\#\mu \,\, {\rm on}\,\, C(Y)\}.$  It is clear that if $T^-$ is standard, then $\S:={\cal D}$ is a closed and convex gambling house. 

Conversely, if $\S$ is a closed gambling house, then (\ref{homo1}) defines a standard map that satisfies properties (1)---(4) of a backward Kantorovich operator. It remains to show that for a fixed $g\in C(Y)$, $T^-g$ is upper semi-continuous. For that assume $x_n\to x$ in $X$. Assuming $\limsup_{n \to \infty}T^-g(x_n)>-\infty$ (since otherwise there is nothing to prove), consider a subsequence such that 
$\limsup_{n \to \infty}T^-g(x_n) = \lim_{j \to \infty}T^-g(x_{n_j})$. Since $\S$ is closed, the set $\S_j:=\{\sigma \in \P(Y); (x_{n_j}, \sigma)\in \S \}$ is compact in $\P(Y)$, hence there is $\nu_j$ in $\S_j$ such that 
$$T^-g(x_{n_j})=\sup\{\int_Xg\, d\sigma; (\delta_{x_{n_j}}, \sigma)\in \S\}=\int_Xg\, d\nu_j.$$ 
 By the weak$^*$ compactness of $\mcal{P}(Y)$, we may extract a further subsequence if necessary and assume $\nu_{j} \to \bar{\nu}$ for some $\bar{\nu} \in \mcal{P}(Y)$. Since $\S$ is closed, we have $(\delta_x, \bar \nu)\in \S$.  It follows that 
\as{
\limsup_{n \to \infty}T^-g(x_n) = \lim_{j \to \infty}T^-g(x_{n_j}) = \int_{Y}gd\bar{\nu} 
\leq \sup_{\nu \in \mcal{P}(Y)}\{\int_{Y}gd\nu; (\delta_x, \nu) \in \S\} = T^-g(x),
}
hence $T^- g \in USC(X)$.

It is clear that $\S$, hence $\overline{\rm conv (\S)}$ is contained in ${\cal D}:=\{(\mu, \nu)\in \mcal{P}(X)\times\mcal{P}(Y); \, \nu \leq T_\#\mu \,\, {\rm on}\,\, C(Y)\}.$ Conversely, by the Hahn-Banach theorem, every $(\delta_x, \nu)$ in $ {\cal D}$ belongs to $\S$, hence $\overline{\rm conv (\S)}={\cal D}$ since the extreme points of the latter are of the form  $(\delta_x, \nu)$ for some $x\in X$ and $\nu \in \P(Y)$. 

\begin{theorem}
 Let ${\mcal S}$ be a non-empty subset of $\mcal{P}(X)\times\mcal{P}(Y)$. The following are then equivalent:
\begin{enumerate}
\item ${\mcal S}$ is a convex weak$^*$-closed gambling house.
\item ${\mcal S}$ is a standard backward transfer set.
\item $\S=\{(\mu, \nu) \in \P(X)\times \P(Y);\, \nu \leq T_\#\mu \}$, where $T:C(Y)\to USC_f(X)$
is a positively $1$-homogenous Kantorovich operator 
and $T_\#\mu (g):=\int_XTg \, d\mu$ for every $g\in C(Y)$. 
\item $\S=\{(\mu, \nu) \in \P(X)\times \P(Y);\, \mu \prec_{\mcal A} \nu\}$, where ${\mcal A}$  is a  balayage cone on  $X \sqcup Y$.
\item 
$\S$ is convex weak$^*$-compact in $\mcal{P}(X)\times\mcal{P}(Y)$ and $(\mu, \nu) \in {\mcal S}$ if and only if there exists $\pi \in {\mathcal K}(\mu, \nu)$ such that  
$(\delta_x, \pi_x) \in {\mcal S}$ for $\mu$-almost $x\in X$, where $(\pi_x)_x$ is a disintegration of $\pi$ with respect to $\mu$. 

\end{enumerate}
\end{theorem}
\noindent{\bf Proof:}
Note that 1), 2) and 3) follow from the above proposition, while the restricted Strassen theorem mentioned in the introduction yields that 4) implies 5). In order to show that 5) implies 2), we compute the Legendre transform of $\mcal{B}_\mu$ where $\mcal{B}$ is the characteristic function  of $\S$ and show that it is a backward linear transfer with Kantorovich operator  
 \eq{ T^- g(x) := \sup_{(\delta_x, \sigma)\in \S}\int_{Y}g(y)d\sigma(y).
}
Indeed, $\mcal{B}_\mu^*(g) = \sup\lf\{\int_{Y}g d\nu \,;\ \nu \in \mcal{P}(X),\, (\mu, \nu)\in \S \rt\}
$
while Strassen's theorem gives that for any $(\mu, \nu)\in \S$, there is $(x,\pi_x)\in \S$ such that 
\eqs{
\int_{Y}g(y) d\nu(y) = \int_{Y}\lf[\int_{X}g(y)d\pi_x(y)\rt]d\mu(x)
\leq \int_{X}\lf(\sup_{(\delta_x, \sigma)\in \S}\int_{Y}g(y)d\sigma(y)\rt) d\mu(x).
}
hence $\mcal{B}_\mu^*(g) \leq \int_XT^- g(x)\, d\mu.$ 

On the other hand,  for each $x \in X$, the supremum of $\sigma \to \int_{X}g(y)d\sigma(y)$ is achieved on the set $\{\sigma \in {\cal P}(X);\,(\delta_x,\sigma)\in \S\}$ since the latter is weak$^*$ closed in $\mcal{P}(X)$ and therefore is weak$^*$ compact. By a standard 
selection theorem,
there is a measurable selection $x\to \sigma_x$, where for each $x$, $ \sigma_x$ is where the maximum is achieved. Note that 
since $\delta_x\prec_\A \sigma_x$, we have that $\mu\prec_A \tilde \nu$,
where 
$\tilde \nu(A) := \int_{X}\sigma_x(A)d\mu(x)$. It follows that
\eqs{
\int_XT^- g(x)\, d\mu\leq\int_{X}\lf(\sup_{(\delta_x, \sigma)\in \S}\int_{X}g(y)d\sigma(y)\rt) d\mu(x)=\int_{X}g(y) d\tilde \nu(y) \leq \mcal{B}_\mu^*(g).
}
Note that the above also show that $D_1(\cal B)=\P(X)$, that is $\cal S$ is a gambling house, where 
\[
D_1(\T)=\{\mu\in \P(X); \exists \nu\in \P(Y), (\mu, \nu)\in D(\T)\}.
\]  

It remains to define a suitable balayage cone. For that, we establish the following general result. 

\begin{theorem}\label{finite}
Let $\T$ be a standard backward linear transfer on $\P(X)\times \P(Y)$ with backward Kantorovich operator $T^-$, then  $D(\T) \subset D(\T_r)$, where $\T_r$ is a 
backward linear transfer on $\P(X)\times \P(Y)$ whose Kantorovich operator is given by the positively $1$-homogenous recession operator associated with $T^-$, i.e.,  
\begin{equation}
T_r^-g(x) := \lim_{\lambda \to +\infty}\frac{T^-(\lambda g)(x)}{\lambda}.
\end{equation}
Moreover, there is a balayage cone $\A$ on $X \sqcup Y$ such that the following are equivalent:
\begin{enumerate}
\item $(\mu, \nu)\in D(\T_r)$.
\item $\mu \prec_\A \nu$.
\item $\int_XT_r^-g\, d\mu \geq \int_Yg\d\nu$ for every $g\in C(Y)$. 
\end{enumerate}
Furthermore, $D(\T_r)$ is the smallest transfer set containing $D(\T)$, hence $D(\T)$ is a transfer set if and only if $D(\T_r)=D(\T)$. 
\end{theorem}
\begin{proof}
Without loss of generality, we may assume $\T \geq 0$ (otherwise consider $\T-C$ where $C \in \R$ is a lower bound for $\T$). Let
\begin{equation}
\T_{\text{finite}}(\mu,\nu) := \begin{cases} 0 & \text{if } (\mu,\nu) \in D(\T),\\
+\infty & \text{otherwise.}
\end{cases}
\end{equation}
It is immediate that $\T_{\text{finite}}$ is a proper, bounded below, convex, and weak$^*$ lower semi-continuous function with $D_1(\T_{\text{finite}}) = D_1(\T)$. We have for $\mu \in D_1(\T_{\text{finite}})$,
\eq{\lbl{tau_finite}
(\T_{\text{finite}})_{\mu}^*(g) = \sup\{\int_{Y}gd\nu\,;\, \nu \in \mcal{P}(Y),\, \T(\mu, \nu) < +\infty\}.
}
Since
 for $\lambda > 0$, we have 
\eqs{
\int_{X}\frac{T^-(\lambda g)}{\lambda} d\mu  = \sup_{\sigma \in {\mathcal P}(Y)}\{\int_Yg\, d\sigma-\frac{1}{\lambda}\T(\mu, \sigma) \}, 
}
it follows that $\liminf_{\lambda \to +\infty}\int_{X}\frac{T^-(\lambda g)}{\lambda}d\mu  \geq \sup\{\int_{Y}gd\sigma\,;\, \T(\mu, \sigma) < +\infty\} = (\T_{\text{finite}})_{\mu}^*(g)$. On the other hand, since $\T$ is non-negative, we have
\eqs{
\limsup_{\lambda \to +\infty}\int_{X}\frac{T^-(\lambda g)}{\lambda}d\mu \leq \sup\{\int_Yg\, d\sigma\,;\, \sigma \in {\mathcal P}(Y),\, \T(\mu, \sigma) < +\infty\} = (\T_{\text{finite}})_{\mu}^*(g),
}
so we conclude that for $\mu \in D_1(\T_{\text{finite}})$,
\eq{\lbl{limit_finite}
(\T_{\text{finite}})_{\mu}^*(g) = \lim_{\lambda \to +\infty}\int_{X}\frac{T^-(\lambda g)}{\lambda}d\mu.
}
To show that $(\T_{\text{finite}})_{\mu}^*(g)=\int_XT_r^-g\, d\mu$,  
it remains to justify that 
\begin{equation*}
\lim_{\lambda \to +\infty} \int_X \frac{T^- (\lambda g)}{\lambda} d\mu = \int_{X}\lim_{\lambda \to +\infty}\frac{T^- (\lambda g)}{\lambda} d\mu.
\end{equation*} 
This is simply by monotone convergence, since as $\T$ is non-negative, the function $\lambda \to \frac{T^- (\lambda g)(x)}{\lambda}$ is monotone increasing, i.e.,  if $\lambda_2 \geq \lambda_1$, 
\as{
\frac{T^-(\lambda_2 g)(x)}{\lambda_2} = \sup_{\nu \in \mcal{P}(Y)}\{\int_{Y}g d\nu - \frac{1}{\lambda_2}\T(\delta_x,\nu)\} 
\geq \sup_{\nu \in \mcal{P}(Y)}\{\int_{Y}g d\nu - \frac{1}{\lambda_1}\T(\delta_x,\nu)\} = \frac{T^-(\lambda_1 g)(x)}{\lambda_1}.
}
Note now that $T^-_r$ is a positively $1$-homogenous Kantorovich operator, hence by Proposition \ref{kanto}, the functional defined on $\P(X)\times \P(Y)$ by 
\[
\T_r(\mu, \nu)=\sup\{\int_Ygd\mu-\int_XT^-_rgd\nu; g\in C(Y)\}
\]
and $+\infty$ outside $\P(X)\times \P(Y)$, is a backward linear transfer that can only take values $0$ or $+\infty$. Moreover,
$
D(\T)\subset D(\T_r)
$
since if $\T(\mu, \nu)<+\infty$, then $\mu\in D_1(\T_{\rm finite})$ and 
\[
0=\T_{\rm finite}(\mu, \nu)= (\T_{\text{finite}})_{\mu}^{**}(\nu)= \sup_{g\in C(Y)}\{\int_Yg\, d\nu-\int_XT_r^-g\, d\mu\},
\]
hence for all $g\in C(Y)$, we have 
$\int_XT^-_rg, d\mu\geq \int_Yg\d\nu$ and therefore $\T_r(\mu, \nu)\leq 0$.

Consider now the cone 
\eq{\mathcal{A} := \{ \phi \in LSC(X \sqcup Y)\,;\, T_r^-(-\phi_Y) \leq -\phi_X\}=\{ \phi \in LSC(X \sqcup Y)\,;\, T_r^+(\phi_Y) \geq \phi_X\},\label{cone}}
 where $\phi_X$ (resp., $\phi_Y$) is the restriction of $\phi$ to the component $X$ (resp., $Y$) and $T_r^+g=-T^-_r(-g)$. Then $\A$ is a balayage cone and the following holds:
 \[
\hbox{ $(\mu,\nu) \in D(\T_r)$ if and only if $\mu \preceq_{\mathcal{A}}\nu$.}
\]
Indeed, the convexity, lower semi-continuity and positive homogeneity of $T_r^-$ clearly yield that $\A$ is a closed convex cone of $LSC(X \sqcup Y)$ containing the constants. The fact that $T_r^-$ is monotone increasing yields that if $\phi^1, \phi^2$ are in $\A$, then
\[
T_r^-(-(\phi^1_Y\vee \phi^2_Y))=T_r^-((-\phi^1_Y)\wedge (- \phi^2_Y))\leq T_r^-(-\phi^1_Y)\wedge T_r^-(- \phi^2_Y)\leq (-\phi_X^1)\wedge (-\phi_X^2)=-(\phi_X^1\vee \phi_X^2).
\]
It follows that $\A$ is a balayage cone.

Suppose now that $(\mu,\nu) \in D(\T_r)$, then
\begin{equation}
0 = \T_r(\mu,\nu) = \sup_{g}\{\int_{Y}gd\nu - \int_{X}T_r^-(g)d\mu\},
\end{equation}
hence $\int_{X}T_r^-(g)d\mu \geq \int_{Y}g d\nu$ for all $g \in C(Y)$. 
Take now any $\phi \in \A$,
then by assumption, 
\begin{equation}
\int_{X}-\phi_X d\mu \geq \int_{X}T_r^-(-\phi_Y)d\mu \geq \int_{Y} (-\phi_Y)d\nu,
\end{equation}
 which translates into $\int_X\phi_X d\mu \leq \int_Y\phi_Y d\nu$ for every $\phi \in {\mathcal A}$, hence $\mu \prec_{_{{\mathcal A}}}\nu$.
 
Conversely, suppose $\mu \preceq_{\mathcal{A}}\nu$. For $g \in C(Y)$, consider the function $\phi \in LSC(X \sqcup Y)$ defined by $\phi_X:= -T_r^-g$ and $\phi_Y := -g$. Then $\phi \in \mathcal{A}$, so that
\begin{equation}
-\int_XT_r^-gd\mu =\int_X\phi_X d\mu \leq \int_Y\phi_Y d\nu= -\int_Y g d\nu,
\end{equation}
 hence $\int_XT_r^-gd\mu \geq \int_Y g d\nu$, which implies that $\T_{r}(\mu,\nu) = 0$ and so $(\mu,\nu) \in D(\T_r)$.
  It follows 
   that $D(\T_r)$ is a balayage, hence a transfer set.

Suppose now that  $\S$ is a transfer set containing $D(\T)$, then its Kantorovich operator is given by
\[
T_\S f(x)=\sup\{\int_Yg\, d\sigma; (\delta_x, \sigma)\in \S\} \geq \sup\{\int_Yg\, d\sigma; (\delta_x, \sigma)\in D(\T)\}=T_r^-f(x).
\] 
If $(\mu, \nu)\in D(\T_r)$, then 
\[
0=\sup \{\int_Yg\, d\mu -\int_X T_r^-g\, d\nu; g\in C(Y)\}\geq \sup\{\int_Yg\, d\mu -\int_XT_\S g\, d\nu, g\in C(Y)\},
\]
which means that $(\mu, \nu)\in \S$ and therefore $D(\T)\subset D(\T_r)\subset \S$. \end{proof}

Now we explore when the domain $D(\T)$ of a linear transfer is a transfer set.

\begin{proposition} Let $\T$ be a standard backward linear transfer, $T^-$ the corresponding Kantorovich operator, and $T^-_r$ the associated recession operator. Then the following properties are equivalent:
\begin{enumerate}

\item $D(\T)$ is a transfer set.

\item $D(\T)=D(\T_r):=\{(\mu, \nu)\in \P(X)\times \P(Y); \int_Yg d\nu \leq \int_XT^-_rgd\mu, \forall g\in C(Y)\}$. 
\item $D(\T)$ is weak$^*$-closed in $\P(X)\times \P(Y)$.

\item $D_1(\T)$ is weak$^*$-closed in $\P(X)$.

\item $D_1(\T)=\P(X)$.

\end{enumerate}
\end{proposition}
\prf{That (1) and (2) are equivalent follows from the preceding proposition. 2) yields 3) since $T^-_rg$ is upper semi-continuous for every $g\in C(Y)$. 3) implies 4) is immediate. For 5) note that ${\cal E}:=\{\delta_x; x\in X\}$ is the set of extreme points of $\P(X)$ and ${\cal E}\subset D_1(\T)$ since $\T$ is standard. The rest follows from Krein-Milman's theorem and the fact that  $D_1(\T)$ is closed and convex. 

If now (5) holds, then by the above proposition, we have 
$(\T_{\text{finite}})_{\mu}^*(g)=\int_XT_r^-g\, d\mu$,  for any $\mu \in D_1(\T_{\text{finite}})=\P(X)$. It follows that 
$D(\T)=D(\T_r)$, hence it is a transfer set. 
}
\rem{Note that in general $D(\T) \neq D(\T_r)$ and $D(\T)$ is not necessarily a transfer set, since we cannot say much about $(\T_{\rm inf})_\mu^*$ when $\mu\notin D_1(\T)$.  However, suppose for every $\mu\in D_1(\T)$, there exists $\alpha \in L^1(\mu)$ such that 
\begin{equation}\label{bound}
\T(x, \sigma) \leq \alpha (x) \hbox{ for all $\sigma \in \P(Y)$ such that $\T(x, \sigma)<+\infty$},
\end{equation}
then $D(\T)=D(\T_r)$ and $D(\T)$ is a transfer set. Indeed, in this case 
\[
T^-g(x)=\sup\{\int_Y g\, d\sigma- \T(x, \sigma); \sigma \in \P(Y)\}\geq \sup\{\int_Y g\, d\sigma; \T(x, \sigma)<+\infty\}-\alpha(x)\geq 
T_r^-g(x)-\alpha(x),
\]
from which follows that if $\T_r(\mu, \nu)<+\infty$, then $\T(\mu, \nu)\leq \T_r(\mu, \nu)+\int_X\alpha(x)\, d\mu<+\infty.$

Conversely, if $D(\T)$ is a transfer set then we shall see that in this case 
\begin{eqnarray}
D(\T)&=&\{(\mu,\nu)\in \P(X)\times \P(Y); \exists \pi\in{\mathcal K}(\mu, \nu),  \T(\delta_x, \pi_x)<+\infty\,\,  \mu{\rm -a.s}\}\label{taut1}\\
&=&\{(\mu,\nu)\in \P(X)\times \P(Y); \exists \pi\in{\mathcal K}(\mu, \nu),  \int_X\T(\delta_x, \pi_x)\, d\mu (x)<+\infty \}\label{tautbis},
\end{eqnarray}
which indicates  a type of uniform boundedness on  $\T$ that is slightly weaker than  (\ref{bound}).
}

\section{Transfer sets and true balayage of probability measures}

 We now consider the case where we can have a true balayage. 
 
 \begin{theorem}\label{true.balayage1} Let ${\mcal S}$ be a standard backward transfer subset of $\P(X)\times \P(X)$ and let $T^-$ be its associated Kantorovich operator. Then 
 \begin{enumerate}
 \item $\S$ is contained in a true balayage set ${\hat \S}=\{(\mu, \nu) \in \P(X)\times \P(X);\, \mu \prec_{\bar {\mcal C}} \nu\}$, where $\A$ is the balayage cone of lower semi-continuous $T^+$-superharmonic functions on $X$. \\
 The Kantorovich operator associated to $\hat \S$, is given for every $f\in C(X)$ and every $x\in X$ by the expression,
 \begin{equation}
 \hat Tf(x):= 
 \inf \{ \varphi (x); \, -
 \varphi \in \A \, \hbox{and $\varphi \ge f$ on $X$}\}.
 \end{equation}

  \item ${\hat \S}$ is minimal in the sense that if ${\mathcal R}$ is any other true balayage set containing $\S$, then ${\mathcal R}\supset {\hat \S}$.
  
 \end{enumerate}   
\end{theorem}
\begin{proof} Since ${\mathcal S} \subset {\mathcal P}(X)\times {\mathcal P}(X)$ is a backward transfer set, we let $\T$ be its characteristic function in such a way that ${\mathcal S}=D(\T)$, with backward Kantorovich operator 
$$ T^- g(x) := \sup_{(\delta_x, \sigma)\in \S}\int_{X}g(y)d\sigma(y),$$
which is positively $1$-homogenous. According to Theorem \ref{finite}, there is also a cone ${\mathcal C}\subset LSC(X \sqcup X)$ such that 
$$
\S=D(\T)=\{(\mu, \nu)\in \mcal{P}(X)\times \mcal{P}(X);\,  \mu \prec_{\mcal{C}} \nu\},$$ 
and that $\mu \prec_{\mcal{C}} \nu$ if and only if $\int_XT^-gd\mu \geq \int_Xg d\nu$ for every $g\in C(X)$. Consider the cone 
\[ \hbox{$\bar {\cal C}=\{f\in LSC(X);\, T (-f) \leq -f\}=\{f\in LSC(X);\, T^+ (f) \geq f\}$}
\]
 and  the set 
$
\bar{\mathcal S}:=\{(\mu, \nu)\in {\mathcal P}(X)\times {\mathcal P}(X); \mu \prec_{\bar {\cal C}}\nu\}.
$

Note that $\bar{\mathcal S} \supset {\mathcal S}$, since if $\int_XT^-gd\mu \geq \int_Xg d\nu$ for every $g\in C(X)$, then  for any $f\in \bar {\cal C}$, we have
\[
-\int_X fd\mu\geq \int_XT (-f)d\mu \geq \int_X -f d\nu.
\]
Moreover, $\bar {\cal C}$ is a balayage cone on $X$ and therefore $\bar \S$ is a transfer set containing the diagonal $\{(\mu, \mu); \mu\in {\mathcal P}(X)\}$. 
Let now 
$$\bar T^-f(x) = \sup_{(\delta_x, \sigma)\in \bar \S}\int_{X}f(y)d\sigma(y)$$
be the Kantorovich operator corresponding to $\bar \S$, and consider the cone
 \eq{\A=\{g\in LSC(X);\, g(x) \leq  \int_Xg\, d\nu\,\, \hbox{for all $(\delta_x, \nu)\in \bar \S$}\}.
}
 It is clear that $\A$ is a closed convex balayage cone. We claim that   
 for any $f\in C(X)$, 
\eq{\lbl{envelope}
f(x)\leq \bar T^-f(x)\leq \hat{f}(x):=\inf \{ \varphi (x); \, -
 \varphi \in \mcal{A}\, \hbox{and $\varphi \ge f$ on $X$}\}.
}
Indeed, since $(\delta_x, \delta_x)\in \bar \S$, we have that $\bar T^-f(x)\geq f(x)$. Moreover,  for any $\phi \in -\mcal{A}$, $\phi \geq f$ on $X$, we have  
\eqs{
\bar T^-f(x) = \sup_{(\delta_x, \sigma)\in \S}\int_{X}f(y)d\sigma(y) \leq \sup_{(\delta_x, \sigma)\in \bar \S}\int_{X}\phi(y)d\sigma(y) \leq \phi(x),
}
the last inequality holding since $\phi \in -\mcal{A}$. Hence (\ref{envelope}) is verified.

Now the mapping $f\to \hat f$ is a backward Kantorovich operator, hence the functional 
\[
\hat \T(\mu, \nu):=\sup_{f\in C(X)}\{\int_Xf\, d\nu -\int_X\hat f\, d\mu\}
\]
is a backward linear transfer, and since $f\to \hat f$ is positively $1$-homogenous, $\hat \T$ is necessarily $\{0, +\infty\}$-valued. We let ${\hat \S}:=D(\hat \T)$ and claim that 
\eq{
{\hat \S}=\{(\mu, \nu) \in \P(X)\times \P(X);\, \mu \prec_{\mcal A} \nu\}. 
}
Indeed, if $(\mu, \nu)\in \hat \S$, then $\int_X\hat f\, d\mu\geq \int_Xf\, d\nu$ for every $f\in C(X)$, hence for every $\psi\in \A$, we have $\int_X-\psi d\mu=\int_X\widehat {(-\psi)} d\mu \geq \int_X-\psi d\nu$, that is  $\int_X\psi d\mu\leq \int_X\psi d\nu$, hence $\mu \prec_\A \nu$. 

For the reverse implication assume $(\mu, \nu)\notin \hat \S$, then 
\[
\sup_{f\in C(X)}\{\int_Xfd\nu-\int_X{\hat f}d\mu\}=+\infty,
\]
hence there is $f\in C(X)$ such that 
\eq{\lbl{big}
\int_X{\hat f}d\nu\geq \int_Xfd\nu\geq \int_X{\hat f}d\mu +1.
} 
Note now that for every $\epsilon > 0$, and for each $x \in X$, we may choose $\phi_{\epsilon, x} \in -\mcal{A}$, $\phi_{\epsilon, x} \geq f$, such that $\phi_{\epsilon, x}(x) \leq {\hat f} (x) + \frac{\epsilon}{2}$. By continuity, there exists an open neighbourhood $B_{r_x}(x) \subset X$ such that $\phi_{\epsilon, x}(x') \leq f(x') + \epsilon$ for all $x' \in B_{r_x}(x)$. The collection $\{B_{r_x}(x)\}_{x \in X}$ is an open cover of $X$, hence by compactness, there exists a finite subcover $\{B_{r_{x_i}}(x_i)\}_{i = 1}^{n}$ of $X$.
Define 
\eqs{
\phi_\epsilon(x) := \min\{\phi_{\epsilon, x_1}(x), \ldots, \phi_{\epsilon, x_n}(x)\}
}
It follows that $\phi_\epsilon \in -\mcal{A}$ since $\mcal{A}$ is closed under maxima.  Moreover, each $x \in X$ belongs to $B_{r_{x_i}}(x_i)$ for some $i \in \{1,\ldots, n\}$, hence
\eqs{
{\hat f}(x) \leq \phi_\epsilon(x) \leq {\hat f}(x) + \epsilon\quad \text{for all $x \in X$.}
}
Combining this with (\ref{big}), we get 
$$\int_X\phi_\epsilon  d\nu\geq \int_X{\hat f}d\nu\geq \int_X{\hat f}d\mu +1\geq \int_X\phi_\epsilon  d\mu-\epsilon +1.$$
In other words, there is a function $\psi\in \A$, namely $\psi=-\phi_\epsilon$ such that $\int_X\psi  d\mu> \int_X\psi  d\nu$, which contradicts the fact that $\mu \prec_\A\nu$.  

We now show that 
\eq{\lbl{=}
\hbox{$\bar \S = {\hat \S}$, $\A=\bar {\cal C}$, and $\bar T^-f = \hat f$ for all $f\in C(X)$.}
}
Indeed,  $\bar \S\subset {\hat \S}$, since $\bar T^-f\leq \hat f$, hence $\hat \T \leq \bar \T$, $\bar \S=D(\bar \T)\subset D(\hat T)=\hat \S$ and $\bar Tf \leq \hat f$.\\
 On the other hand, if $\phi \in \bar {\cal C}$, then $\phi (x)\leq \int_X\phi \, d\mu$ for every $(\delta_x, \mu)\in \bar \S$, hence $\bar {\cal C} \subset \A$, from which follows that $\bar T^-f \geq \hat f$ and $\hat \S \subset \bar S$. Therefore, $\A=\bar {\mathcal C}$, the latter being the cone of $T^+$-superharmonic functions, and   
  (\ref{=}) is established.

(2) To show that $\bar S$ is minimal, assume $\cal B$ is a balayage cone in $C(X)$ such that 
\[
{\mathcal R}:=\{(\mu, \nu)\in \P(X)\times \P(X); \mu \prec_{\mcal B} \nu\},
\]
and $\S \subset \cal R$. For every $\phi \in {\mcal B}$, we then have $\phi (x)\leq \int_X\phi \, d\mu$ for every $(\delta_x, \mu)\in \S \subset \bar\S$, hence ${\mcal B} \subset \A$, from which follows that $\hat \S =\bar \S\subset \cal R$.
\end{proof}

 \begin{remark} There are actually several other transfer sets between $\S$ and $\hat \S$. Indeed, let 
 \[
 \hbox{$T_1f:=T^-f \vee f$, \quad $T_2f:=\overline{\lim_n\uparrow T_1^n f}$, \quad and $T_3f:=\lim_n\downarrow T_1^n \hat f$.
 }
 \]
 They are all positively $1$-homogenous Kantorovich operators such that
 \eq{\lbl{many}
 T^- f\leq T_1f \leq T_2f \leq  T_3f\leq \bar Tf=\hat f. 
 }
 Indeed, the relation between $T, T_1$ and $T_2$ is immediate. Also,  for any $f\in C(X)$, we have  
  \[T^-f(x)= \sup_{(\delta_x, \sigma)\in \S}\int_{X}f(y)d\sigma(y)\leq \sup_{(\delta_x, \sigma)\in {\bar \S}}\int_{X}f(y)d\sigma(y)=\bar T^-f(x) = \hat f(x), 
  \] 
  hence $T_1f\leq \hat f$, and therefore $T_1\hat f\leq \hat f$, $T_1^{n}f \leq \hat f$, and   $T_1^{n+1}\hat f \leq T^{n}\hat f \leq \hat f$, since $\bar Tf=\hat f$ is idempotent. It follows that $T_2$ and $T_3$ are well defined, are Kantorovich operators, and are both below $\bar Tf=\hat f$. Finally, note that $T_1^nf \leq T_1^n\hat f$, hence the inequalities in (\ref{many}) hold. 
  
For $i=1, 2, 3$, let 
$$\S_i=\{(\mu, \nu)\in \P(X)\times \P(X); \, \nu \leq T_i\# \mu\}, 
$$
where $T_i\#\mu(f)=\int_XT_if d\mu$.
It is clear that they are all transfer sets with
 \[
 \S \subset \S_1  \subset \S_2  \subset S_3  \subset \hat S,
 \] 
 where $\S_1$ amounts to adding the diagonal to $\S$, while $\hat S=\{(\mu, \nu)\in \P(X)\times \P(X); \nu \leq \hat \mu\}$, with the familiar sublinear functional $ \hat \mu (f)=\mu (\hat f)$ from Choquet theory. 
 
 Finally, we note why the process does not necessarily stop at $T_2$. Indeed, by setting 
 $U_\infty f:=\lim_n\uparrow  T_1^n f$, we can use that $T_1$ is a functional capacity to deduce that  $T_1 (U_\infty f)=U_\infty f$ and  $U_\infty \circ U_\infty f=U_\infty f$. It is clear that $U_\infty f \leq \hat f$. On the other hand, since $U_\infty f=U (U_\infty f)=T^-(U_\infty f) \vee U_\infty f$, it follows that $T^-(U_\infty f)\leq U_\infty f$.
 
 The process would stop if $U_\infty f$ were upper semi-continuous, since then $U_\infty f \in -\bar {\cal C}$, and since $U_\infty f \geq f$, it would follow that $U_\infty f \geq \hat f$ and therefore $U_\infty f =\hat f$. This happens for instance if $T^-$ satisfies for any $f\in C(X)$, $x, y \in X$,
 \eq{\lbl{unif}
 T^-f(x)-T^-f(y) \leq \omega(d(x, y)),
 }
 where $d$ is the metric on $X$ and $\omega$ is the modulus of uniform continuity of $f$. This occurs for example in the case where $X$ is the closure of an open bounded domain $O$ in $\C^n$ (resp., $\R^n$), and the Kantorovich operators are of the form 
 \[
T^-f(x):= \sup_{v\in \R^n} \bigg\{ \int ^{2\pi}_0 f(x +
 e^{i\theta}v) {d\theta \over 2\pi};\,  x + \bar \Delta v\subset O \bigg\}, 
 \]
 (resp.,  
 \[
T^-f(x) =\sup_{r \geq 0} \bigg\{ \int_{B} f (x + r y) \,dm(y);\,  x + r \overline{B}\} \subset O \bigg\},
\]
where $\Delta = \{ z \in {\mathbb C}, \vert z\vert < 1\}$ is the open unit disc in $\C$, , 
$B$ is the open unit ball in $\R^n$ centered at $0$, and $m$ is normalized Lebesgue measure on $\R^n$. Condition (\ref{unif}) is verified in either case, and we have $\lim_n\uparrow  T^n f=\hat f$, where in the first case the envelope $\hat f$ corresponds to teh cone of plurisubharmonic functions on $O$, while in the second it corresponds to the cone of subharmonic functions on $O$. In other words, $T=T_1$ and $T_2=T_3=\bar T$. 
However, it is not always the case and we had to add the upper semi-continuous regularization to be able to define $T_2f=\overline{U_\infty f}$. 
\end{remark}

\begin{corollary}\label{true.balayage} Let ${\mcal S}$ be a subset of $\mcal{P}(X)\times\mcal{P}(X)$. Then the following are equivalent:
\begin{enumerate}
\item ${\mcal S}$ is a transitive backward transfer set containing the diagonal $\{(\mu, \mu); \mu\in {\mathcal P}(X)\}$. 

\item  ${\mcal S}$ is a backward transfer set with a Kantorovich operator $T$ satisfying $T^2=T\geq I$.

\item $\S=\{(\mu, \nu) \in \P(X)\times \P(X);\, \mu \prec_{\mcal A} \nu\}$, where $\A$ is a balayage cone  in $LSC(X)$.  
\end{enumerate}
\end{corollary}
\begin{proof} That 3) implies 1) follows from Proposition \ref{transfer.set} with the two facts that the diagonal is now contained in $\S$ and that the latter is transitive, following readily from the true balayage relation with respect to a cone in $LSC(X)$.

To show that 1) implies 2), we consider for all $(\mu, \nu)\in \P(X)\times \P(X)$, the functional 
\[
\T_2(\mu, \nu)=\inf\{\T(\mu, \sigma)+\T(\sigma, \nu); \sigma \in \P(X)\}.
\]
It is a backward linear transfer with Kantorovich operator $T^2=T\circ T$ by Proposition \ref{conv}. Since $\T$ is transitive, it is clear that $D(\T_2)\subset D(\T)$. On the other hand, since  $\T_2(\mu, \nu)\leq \T(\mu, \mu)+\T(\mu, \nu)$ and  $D(\T)$ contains the diagonal, we have $D(\T)\subset D(\T_2)$. Since both $\T$ and $\T_2$ are valued in $\{0, +\infty\}$, it follows that $\T=\T_2$ and consequently, $T=T^2$.

To show that 2) implies 3), we use Theorem \ref{true.balayage1} to write $\S \subset \hat \S$, where $\hat \S=\{(\mu, \nu) \in \P(X)\times \P(X);\, \mu \prec_{\mcal A} \nu\}$, and  
$\A=\{g\in LSC(X);\, g(x) \leq  \int_Xg\, d\nu\,\, \hbox{for all $(\delta_x, \nu)\in \S$}\}$.  Moreover, $T f(x) \leq \hat{f}$. 
Since $Tf(x)=T^2f(x)= \sup\limits_{(\delta_x, \sigma)\in \S}\int_{X}Tf(y)d\sigma(y)$, it follows that $Tf\in -\A$, and since $Tf \geq f$, we conclude that $Tf\geq \hat f$. In other words, $Tf= \hat f$ and $\S=\hat \S$.
\end{proof}
The following is now immediate.
\begin{corollary} Let $T: C(Y) \to C(X)$ be a Markov operator (i.e., $T$ positive linear continuous and $T1=1$), and consider the set 
  \[
   { \cal S}:=\{(\mu, \nu)\in {\mathcal P}(X) \times {\mathcal P}(Y);\, \nu = T^*(\mu)\},
    \]
  where $T^*:{\mathcal M}(X) \to {\mathcal M}(Y)$ is the adjoint operator. Then, 
  \begin{enumerate}
  \item $\S$ is a backward linear transfer, whose Kantorovich operator is $T$ itself.
  \item $\S$ cannot be a true balayage set with respect to a cone $\A$ in $C(X)$, unless $X=Y$ and $T$ is the identity.
  \item $\S$ is a restricted balayage set for the cone 
$
\A=\{(f, g)\in LSC(X \sqcup X); f\leq Tg\}.
$
\item The smallest true balayage set containing $\T$ is the one corresponding to the cone of superharmonic functions, i.e.,
$
\bar {\cal C}=\{f\in LSC(X); f\leq Tf\}.
$
\end{enumerate} 
\end{corollary}

\begin{corollary} Let $\T$ be a standard backward linear transfer on $\mcal{P}(X)\times \mcal{P}(X)$ 
and let $T$ be the corresponding Kantorovich operator. Then,
\begin{enumerate}
\item There exists a true balayage cone $\A$ in $C(X)$ such that  
 $$D(\T)\subset D(\T_r)\subset \widehat{D(\T_r)}=\{(\mu, \nu)\in \mcal{P}(X)\times \mcal{P}(X); \mu \prec_{\mcal{A}} \nu\},$$ 
 where $D(\T_r)$ is the domain of the recession operator $T_r$ of $T$. 
 
\item $\widehat{D(\T_r)}$ is the smallest true balayage set containing $D(\T)$.

\item If $D(\T)$ is a transitive transfer set containing the diagonal, then $D(\T)=D(\T_r)=\widehat{D(\T_r)}$ and the true balayage with respect to $\A$ is equivalent to the restricted balayage associated with $T_r^-$. 
\end{enumerate}
\end{corollary}

\begin{proof} By Theorem \ref{finite}, $D(\T_r)$ is the smallest  transfer set containing $D(\T)$, with the recession operator $T_r^-$ being the corresponding Kantorovich operator for $\T_r$. The transfer set $\widehat{D(\T_r)}$ is the balayage set associated to $D(\T_r)$ by Theorem \ref{true.balayage1}. 

If $D(\T)$ is a transitive transfer set, then by the minimality properties of both $D(\T_r)$ and $\widehat{D(\T_r)}$, we obtain that $D(\T)=D(\T_r)=\widehat{D(\T_r)}$. By the uniqueness of the Kantorovich operator associated to a transfer set, we have $T^-_rg={\hat g}$. 

Finally, consider the cone ${\mathcal C}\subset LSC(X \sqcup X)$ such that 
$
D(\T_r)=\{(\mu, \nu)\in \mcal{P}(X)\times \mcal{P}(X);\,  \mu \prec_{\mcal{C}} \nu\},$  
and that $\mu \prec_{\mcal{C}} \nu$ if and only if $\int_XT_r^-gd\mu \geq \int_Xg d\nu$ for every $g\in C(X)$. We claim that 
$\mu \prec_{\mcal{A}} \nu$
if and only if  $ \mu \prec_{\mcal{C}} \nu$,
 which means that the restricted balayage is equivalent to the true balayage. Indeed, if  $\mu \prec_{\mcal{A}} \nu$, i.e. $\int_{X} \vphi d\nu \leq \int_{X}\vphi d\mu$ for all functions $\vphi$ in $-\mcal{A}$, then for all $g \in C(X)$,
  \eqs{
  \int_{X}gd\nu \leq \int_{X}\hat{g}d\nu \leq \int_{X}\hat{g}d\mu = \int_{X}T^-_r g d\mu, 
  }
 hence $ \mu \prec_{\mcal{C}} \nu$. On the other hand, if $\mu \prec_{\mcal{C}} \nu$ then $\int_{X}\hat{g}d\mu = \int_XT_r^-gd\mu \geq \int_X g d\nu$ for every $g\in C(X)$, the inequality holds for all $g\in -\mcal{A}$ in which case $g=\hat g$. This implies $\mu \prec_{\mcal{A}}\nu$. 
\end{proof}

  \section{Representation of linear transfers as optimal balayage transport with cost}\label{representation_optimal_weak_transports}
 In this section we show that every backward linear transfer can be represented as the minimal cost of an optimal balayage transport of measures, by combining our previous results with the following properties of optimal weak transports established in \cite{Go} and \cite{ABC}. We shall sketch a proof of the latter fact using results in \cite{BBP} since we can then exhibit another  interesting representation of a linear transfer as a somewhat classical mass transport of measures but on the enlarged state space $X\times \P(Y)$.   
   
\thm{\label{prop.zero} Let $c\colon X\times \mathcal P(Y) \rightarrow \R\cup\{+\infty\}$ be a  bounded below, lower semi-continuous functional such that for each $x\in X$, $\sigma \to c(x, \sigma)$  is proper and convex. Consider the value functional of an optimal weak transport, i.e,  for a pair $(\mu, \nu)\in \P(X)\times \P(Y)$, 
\eq{\lbl{weak_transport2}
{\mathcal V}_c(\mu,\nu) := \inf_\pi\{\int_X c(x, \pi_x)\, d\mu(x); \pi \in {\mathcal K}(\mu, \nu)\},
}
where $(\pi_x)_x$ is the disintegration of $\pi$ with respect to $\mu$, 
 Then, 
 \begin{enumerate}
\item 
${\mathcal V}_c$ 
is a standard backward linear transfer with corresponding Kantorovich operator given by
 $T^- g(x) := \sup\{ \int_{Y}gd\sigma - c(x,\sigma)\,;\, \sigma \in \mcal{P}(Y)\}$.
\item Conversely, if $\T$ is a standard backward linear transfer,  
then $\T={\mathcal V}_c$, where ${\mathcal V}_c$ is the optimal weak transport associated with the cost function $c(x,\sigma) := \T(\delta_x, \sigma)$.
\end{enumerate} 
Moreover, for any $(\mu, \nu)$ in 
$D(\T)$, there exists a transport $\pi\in  {\mathcal K}(\mu, \nu)$ such that $(\delta_x, \pi_x)\in D(\T)$ for $\mu$-almost $x\in X$ and
\begin{equation} \label{superStrassen2}
\T(\mu, \nu)=\int_X \T(\delta_x, \pi_x)\, d\mu(x).
\end{equation}
Furthermore, $D(\T)$ is a backward transfer set if and only if
\begin{eqnarray}
D(\T)=\{(\mu,\nu)\in \P(X)\times \P(Y); \exists \pi\in{\mathcal K}(\mu, \nu),  \T(\delta_x, \pi_x)<+\infty\,\,  \mu{\rm -a.s}\}.\label{taut}
\end{eqnarray}

}
 \begin{lemma}\label{rep} Under the above hypothesis on the cost $c$, we have ${\mathcal V}_c(\mu,\nu) = \hat {\mathcal V}_c(\mu,\nu)$, 
where 
	\begin{equation}\label{eq:defhatV}
	\hat {\mathcal V}(\mu,\nu):=\inf_{P \in \Lambda(\mu,\nu)} \int_{X\times \mathcal P(Y)}  c(x,p)dP(x,p), 
	\end{equation}
	\begin{equation}
	\Lambda(\mu,\nu) := \left\{ P \in \mathcal P(X\times \mathcal P(Y))\, \mid\, \text{proj}_X P = \mu,\, b(\text{proj}_{\mathcal P(Y)}(P)) = \nu\right\},
\end{equation}
and $b$ is the barycentric map on probability measures. 	
\end{lemma}
 \prf{Recall that for a probability measure $P\in\mathcal P(\mathcal P(Y))$, its 
{\em barycenter} $b(P)$ is the unique probability in $\mathcal P(Y)$ determined by
\begin{align}\label{eq:def intensity1}
	b(P)(g) = \int_{\mathcal P(Y)} p(g) dP(p) \quad \forall g\in C(Y).
\end{align}
For any $\pi\in \mcal{K}(\mu,\nu)$, we let $(\pi_x)_{x\in X}$ be its  disintegration  with respect to $\mu$, and 
consider the measurable map 
$\kappa_\pi \colon  X \rightarrow X \times \mathcal P(Y)$ defined by $x\mapsto (x,\pi_x)$.  
Let now 
$J(\pi):=  (\kappa_\pi)_\#(\mu),$ 
which yields a map from $\mathcal P(X\times Y )$ to $ \mathcal P(X\times\mathcal P(Y))$ such that 
for $\pi\in{\mathcal K}(\mu,\nu)$, we have $J(\pi) \in \Lambda(\mu,\nu)$ and
	$\int_X c(x,\pi_x)\mu(dx) = \int_{X\times \mathcal P(Y)} c(x,p) J(\pi)(dx,dp)$, hence 
	$${\mathcal V}_c(\mu,\nu) =\inf_{\pi\in\ \mcal{K}(\mu,\nu)} \int_X c(x,\pi_x) \mu(dx) \geq \inf_{P\in\Lambda(\mu,\nu)} \int_{X\times \mathcal P(Y)} c(x,p)P(dx,dp)=\hat {\mathcal V}_c(\mu,\nu).$$
	
For the converse, associate to any probability measure $P\in \mathcal P(X\times \mathcal P(Y))$,  its {\em intensity}, which is the unique measure $I(P) \in \mathcal P(X\times Y)$ such that
\begin{align}\label{eq:def intensity2}
\int_{X\times Y} f(x,y) I(P)(dx,dy) = \int_{X\times \mathcal P(Y)} \int_Y f(x,y)p(dy) P(dx,dp)\quad \forall f\in C(X\times Y).
\end{align}
Now letting $P\in\Lambda(\mu,\nu)$, we easily see that $I(P) \in \mcal{K}(\mu,\nu)$ and $I(P)_x = \int_{\mathcal P(Y)} p \,P_x(dp)$ for $\mu$-a.e $x$. Using convexity we get that
	\begin{align*}
	\int_{X\times \mathcal P(Y)} c(x,p) P(dx,dp) &=
	\int_X \int_{\mathcal P(Y)} c(x,p) P_x(dp) \mu(dx) 
	\geq \int_X c(x,I(P)_x) \mu(dx)\\
	&\geq \inf_{\pi\in \mcal{K}(\mu,\nu)} \int_X c(x,\pi_x) \mu(dx).
	\end{align*}
	It follows that ${\mathcal V}(\mu,\nu) = \hat {\mathcal V}(\mu,\nu)$. 
	}	
\lem{\label{lifting}
	Let $c\colon X\times \mathcal P(Y) \rightarrow \R\cup\{+\infty\}$ be as above, then 
	\begin{enumerate}
	\item The map $\pi \to \int_X c(x, \pi_x)\, d\mu(x)$ is weak$^*$-lower semi-continuous and convex on ${\mathcal K}(\mu, \nu)$. 
	\item For $\mu \in \P(X)$, the map $\nu\to {\mathcal V}_c(\mu, \nu)$ is convex and weak$^*$-lower semi-continuous on $\P(Y)$.
\end{enumerate}
}	
\prf{ (1) 
To show lower semicontinuity of the map $\pi \to \int_X c(x, \pi_x)\, d\mu(x)$, let $\pi^k\to \pi$ in ${\mathcal K}(\mu, \nu)$ and set $P^k=J(\pi^k)$. We may assume that $\liminf_k \int_X c(x,\pi_x^k) d\mu = \lim_k \int_X c(x,\pi_x^k)d\mu $ by selecting a subsequence.  Since $\mathcal P(X\times \mathcal P(Y))$ is compact, let $P$ be an accumulation point of $\{P^k\}_k$ and consider a subsequence converging to $P$. Observe that $$\int_X c(x,\pi_x^k) d\mu = \int_{X\times \mathcal P(Y)}c(x,p)\,P^k(dx, dp).$$
Since $P\to \int_{X\times \mathcal P(Y)}c(x,p)\,P( dx, dp)$ is lower semi-continuous, we have 
	$$\liminf_k \int_{X\times \mathcal P(Y)}c(x,p)\, P^k( dx, dp)\geq \int_{X\times \mathcal P(Y)}c(x,p)\,P( dx, dp).$$

The $X$-marginal of $P$ equals the $X$-marginal of $\pi$, which is $\mu$. Letting $P_x( dp)$ be a disintegration of $P$ with respect to $\mu$, we get from the convexity of $c(x,\cdot)$ that 
	\begin{align*}
	\liminf_k \int_X c(x,\pi_x^k)\ d\mu\geq \int_{X}\int_{\mathcal P(Y)}c(x,p)\,P_x( dp)d\mu\geq  \int_X c\Big(x,\int_{\mathcal P(Y)}p( dy) P_x( dp)\Big)d\mu.
	\end{align*}
	Now we note that $\pi_x( dy)= \int_{\mathcal P(Y)}p( dy)\, P_x( dp)$ for $\mu$-a.e. $x$. Indeed, if $f\in C(X\times Y)$, then 
	$$\int_{X\times Y} f(x,y)\pi^k( dx, dy)\to \int_{X\times Y} f(x,y)\pi( dx, dy).$$
	 Moreover, the function $F(x,p):=\int_Y f(x,y)p( dy)$ is also continuous on $X\times \mathcal P(Y)$, hence $\int F dP^k \to \int F dP$. We then deduce $$\int_{X\times Y} f(x,y)\pi( dx, dy) = \int_{X\times \mathcal P(Y)} F dP = \int_{X\times \mathcal P(Y)}\int_Y f(x,y)p(dy)P( dx, dp), $$
hence $\pi_x( dy)= \int_{\mathcal P(Y)}p( dy)\, P_x( dp)$ for $\mu$-almost every $x$. We finally obtain
	$$\liminf_k \int_X c(x,\pi_x^k)d\mu\geq \int_X c(x,\pi_x)d\mu.$$
The convexity follows from the convexity of $c$ in the second variable. 

(2) To prove the convexity of $\nu \to {\mathcal V}_c(\mu, \nu)$,   fix $\nu_1, \nu_2 \in \mcal{P}(Y)$ and find for a fixed $\epsilon > 0$, $\pi^1 \in \mcal{K}(\mu,\nu_1)$ and $\pi^2 \in \mcal{K}(\mu,\nu_2)$ such that 
$
\int_X c(x, \pi^i_x)\, d\mu(x) \leq {\mathcal V}_c(\mu, \nu_i) +\epsilon\,\, \text{for $i=1, 2$.}
$
  Define $\pi \in \mcal{P}(X \times Y)$ via 
  $d\pi(x,y) := (\lambda d\pi^1_x(y) +(1-\lambda)d\pi^2_x(y))d\mu(x).$
  With $\nu_\lambda :=\lambda \nu_1+(1-\lambda)\nu_2$, we have $\pi \in {\cal K}(\mu, \nu_\lambda)$, and therefore, by convexity of $c$ in the second variable, we have
 \begin{align*}
 {\mathcal V}_c(\mu, \nu_\lambda)&\leq \int_X c(x, \pi_x)\, d\mu(x) \leq \int_X \lambda c(x, \pi^i_x)\, d\mu(x) + \int_X (1-\lambda) c(x, \pi^i_x)\, d\mu(x)\\
 & \leq \lambda {\mathcal V}_c(\mu, \nu_1) + (1-\lambda) {\mathcal V}_c(\mu, \nu_2) + \epsilon.
 \end{align*}

 We now show that $\nu \mapsto {\mathcal V}_c(\mu,\nu)$ is weak$^*$ lower semi-continuous. 
  Using Lemma \ref{rep}, it 
 suffices to show that $\nu \mapsto \hat {\mathcal V}_c(\mu,\nu)$ is weak$^*$ lower semi-continuous.
 For that, let $\nu_n \to \nu$ and for $\epsilon >0$, select 
for each $n$, $P^{n} \in \Lambda (\mu, \nu_{n})$ such that $\int_{X\times \mathcal P(Y)} c(x,p)P^n(dx,dp)\leq \hat {\mathcal V}(\mu,\nu_n)+\epsilon$. Now use the compactness of  $\mathcal P(X \times \mathcal P(Y))$ 
to find a further subsequence (which we again relabel to $n$) so that $P^{n} \to P$ for some $P$, which is necessarily in $\Lambda (\mu,\nu)$ since the intensity function is continuous. It follows that 
\[
\hat {\mathcal V}(\mu,\nu) \leq \int_{X\times \mathcal P(Y)} c(x,p)P(dx,dp)\leq \liminf_{n \to \infty} \int_{X\times \mathcal P(Y)} c(x,p)P^n(dx,dp)\leq \liminf_{n \to \infty}\hat {\mathcal V}(\mu,\nu_n)+\epsilon,  
\]
 which concludes the proof of lower semi-continuity. 
}
\noindent{\bf Proof of Theorem \ref{prop.zero}:}
To show that ${\mathcal V}:={\mathcal V}_c$ is a backward linear transfer, note first that the above lemma yields that $\nu \to {\mathcal V}_c(\mu, \nu)$ is convex and lower semi-continuous. 
 We now show that for every $\mu \in \P(X)$, its Legendre transform 
 is $g\to  \int_{X}T^-g d\mu$, where 
\[
T^-g(x)=\sup_{\sigma \in \mcal{P}(Y)}\{\int_{Y}gd\sigma - c(x,\sigma)\}.
\]
Indeed, if $\mu \in D_1({\mathcal V}_c)$, then for $g\in C(Y)$, 
\al{
-\infty <{\mathcal V}^*_\mu (g)&=\sup_{\nu \in \mcal{P}(Y)}\{\int_Yg\, d\nu -  {\mathcal V}(\mu, \nu)\}\no\\
 &=\sup_{\nu \in \mcal{P}(Y)}\sup_{\pi \in \mcal{K}(\mu,\nu)}\{\int_Yg(y)\, d\nu (y) -  \int_Xc(x, \pi_x)\, d\mu(x)\}\no\\
&=\sup_{\pi \in \mcal{K}(\mu,\cdot)}\{\int_X\lf[\int_Yg(y) d\pi_x(y) - c(x,\pi_x)\rt]\, d\mu (x) \}\lbl{sup_plan}\\
&\leq \int_{X} \sup_{\sigma \in \mcal{P}(Y)}\{\int_{Y}gd\sigma - c(x,\sigma)\}d\mu(x)\lbl{weak_inequality}\\
&= \int_{X}T^-g d\mu .\no
}
On the other hand, the lower semi-continuity of $c$ yields that for each $x$ the supremum in \refn{weak_inequality} is achieved by some $\pi_x$ in a measurable way so that
\eqs{
T^-g(x)=\int_Yg(y) d \pi_x(y) - c(x, \pi_x) \,\, \hbox{for every $x\in X$.}
}
Define $\pi \in \mcal{P}(X\times Y)$ via $d\pi(x,y) := d\pi_x(y)d\mu(x)$. Denoting $\nu := {\text{Proj}_Y}_\# \pi$, we have $\pi \in \mcal{K}(\mu,\nu)$. Hence, by using 
\refn{sup_plan}, we get that
\as{
{\mathcal V}^*_\mu (g)&\geq \int_X\lf[\int_Y g(y) d\pi_x(y)\ d\mu (x) - \int_X c(x, \pi_x)\rt] d\mu(x)= \int_X T^-g(x) d \mu (x), 
}
 and   ${\mathcal V}^*_\mu (g)=\int_XT^-g(x) d \mu (x)$, hence ${\mathcal V}_c$ is a backward linear transfer and by (\ref{duality}), we have 
 \[
 {\mathcal V}_c(\mu, \nu)= \sup_{g \in C(X)}\{\int_{X}gd\nu - \int_{X}T^-g d\mu\}\quad \hbox{for all $(\mu, \nu)\in \P(X)\times \P(Y)$}.    
 \] 
 If now $\mu \notin D_1({\mathcal V}_c)$, that is if ${\mathcal V}_c(\mu, \nu)= +\infty$ for every $\nu\in \P(Y)$, then we claim that $\int_XT^-g\, d\mu=-\infty$ for all $g\in C(Y)$. Indeed if not, we have for some $g\in C(Y)$, 
\[
-\infty<\int_XT^-g(x)\, d\mu(x)=\int_X\sup_\sigma\{\int_Ygd\sigma -{\mathcal V}_c(\delta_x, \sigma)\}\, d\mu=\int_X\int_Y\{g(y)d\sigma_x(y)-{\mathcal V}_c(\delta_x, \sigma_x)\}d\mu(x),
\]
where $x \to \sigma_x$ is a measurable selection of where $\sigma \to \sup_\sigma\{\int_Ygd\sigma -{\mathcal V}_c(\delta_x, \sigma)$ is attained on $\P(Y)$. Let now $\nu:=\int_X\sigma_x\, d\mu(x)$ and note that by the definition of ${\mathcal V}_c$, we have ${\mathcal V}_c(\mu, \nu)\leq \int_X{\mathcal V}_c(\delta_x, \sigma_x)d\mu(x)$,
 hence
\[
\int_XT^-g(x)\, d\mu(x)=\int_Yg(y)d\nu(y)-\int_X{\mathcal V}_c(\delta_x, \sigma_x)d\mu(x)\leq \int_Yg(y)d\nu(y)-{\mathcal V}_c(\mu, \nu).
\]
It follows that ${\mathcal V}_c(\mu, \nu)<+\infty$, hence $\mu\in D_1({\mathcal V}_c)$, which is a contradiction. This means that ${\mathcal V}_\mu^*(g)=\int_XT^-gd\mu$ for every $\mu\in \P(X)$ and ${\mathcal V}_c$ is a linear backward transfer.

(2) For the reverse implication, assume $\T$ is a standard backward linear transfer with backward Kantorovich operator $T^-$, and define $c(x,\sigma) := \T(\delta_x, \sigma)$. Then $c$ is bounded below, lower semi-continuous such that $\sigma \mapsto c(x,\sigma)$ is proper and convex. By the first part, ${\mathcal V}_c$ is a standard backward transfer with Kantorovich operator 
$T^-g(x)=\sup_{\sigma \in \mcal{P}(Y)}\{\int_{Y}gd\sigma - c(x,\sigma)\}
 \,\, \hbox{for every $x\in X$,}$
which is the same for $\T$. Hence, 
for  all $(\mu, \nu)\in \P(X)\times \P(Y)$, we have
\[
\T(\mu, \nu)=\sup_{g\in C(Y)}\{\int_Ygd\nu-\int_XT^- gd\mu\}={\mathcal V}_c(\mu, \nu) \quad \hbox{and $\T={\mathcal V}_c$}.   
\]
The attainment in the disintegration (\ref{superStrassen2}) follows from the lower semi-continuity of $\pi \to \int_X \T(\delta_x, \pi_x)d\mu$, which implies that the infimum in (\ref{weak_transport2}) is attained. Note now that 
\[
D(\T)=D({\mathcal V}_c)=\{(\mu, \nu);\exists \pi\in {\mcal K}(\mu, \nu)\, {\rm with}\, (x, \pi_x)\in D(\T)\, {\rm and}\,  \T(\delta_x, \pi_x)\in L^1(\mu)\}, 
\]
and by Theorem \ref{transfer.set}, 
\[
D(\T) \subset D(\T_r)=\{(\mu, \nu);\exists \pi\in {\mcal K}(\mu, \nu)\, {\rm with}\, (\delta_x, \pi_x)\in D(\T_r) \quad \mu-{\rm a.s}\}.
\]
If now $D(\T)$ is a backward transfer, then $D({\mathcal V}_c)=D(\T)=D(\T_r)$, which establishes (\ref{taut}). 

\prf{of Theorem \ref{prop.main1}: To prove that $\mathcal{B}_{c, \A}$ is a backward linear transfer, we consider the zero-cost balayage transfer associated to the balayage cone $\A$, that is
  \begin{equation*}
c_{\A}  (x, \sigma)=\left\{ \begin{array}{llll}
0 \quad  &\hbox{if $ \delta_x \prec_\mcal{A}\sigma$,}\\
+\infty \quad &\hbox{\rm otherwise.}
\end{array} \right.
\end{equation*}
We can then write 
\begin{equation}\label{optimalstrassen}
B_{c, \A} (\mu, \nu):=\inf\{ \int_X\big\{c (x, \pi_x)+  c_{\A} (x, \pi_x)\big\}\, d\mu (x); \pi \in {\mathcal K}(\mu, \nu)\}.
 \end{equation}
 Note that the cost ${\tilde c}(x, \sigma)=c(x,\sigma)+ c_{\A} (x, \sigma)$ is a proper, bounded below, lower semi-continuous functional that is  convex in the second variable. and therefore Theorem \ref{prop.zero} applies to yield that $\mathcal{B}_{c, \A}$ is a backward linear transfer  on $\mathcal{P}(X)\times \mathcal{P}(Y)$ with backward Kantorovich operator 
    \as{
   T^- g(x) = \sup\{\int_{Y}g d\sigma - c(x,\sigma);\, \sigma \in \mcal{P}(Y),\,\delta_x \prec_{\A}\sigma
\}.
}
(ii)  Suppose now $\T$ is a backward linear transfer, then by Theorem \ref{prop.zero}, we can write 
\eqs{\label{optimalStrass}
\T(\mu, \nu) = 
\inf\{\int_X \T(\delta_x, \pi_x)\, d\mu(x); \pi \in {\mathcal K}(\mu, \nu)\}.   
 }
By Theorem \ref{finite}, $D(\T)$ is contained in a balayge set  
$
\S:=\{(\mu, \nu)\in \P(X)\times \P(Y); \mu\prec_\A \nu\},
$
where $\A$ is a  cone on $X \sqcup Y$. This means that if $(\mu, \nu)\in D(\T)$, then there exists $\pi\in K(\mu, \nu)$ whose disintegration $(\delta_x, \pi_x)\in \S$.
Now let $c(x, \sigma)=\T(\delta_x, \sigma)$ and consider the optimal restricted balayage transfer
\eqs{
{\mcal B}_{c,\A}(\mu, \nu) := \begin{cases}
\inf\{\int_X c(x, \pi_x)\, d\mu(x); \pi \in {\mathcal K}_{\mathcal A}(\mu, \nu)\}   
 &\text{if  
 \, $\mu \prec_{\mathcal A} \nu$}\\
+\infty  &\text{otherwise,}
\end{cases}
}
Note that ${\mcal B}_{c,\A}$ can be written as in (\ref{optimalstrassen}) and therefore $\T \leq {\mcal B}_{c,\A}$.

 On the other hand, we claim that ${\mcal B}_{c,\A} \leq \T$. Indeed, it is clearly the case when $(\mu, \nu)\notin D(\T)$. If however $(\mu, \nu)\in D(\T)$, then  $\mu \prec_{\mathcal A} \nu$ and since 
 \[
{\mathcal K}_{\mathcal \T}(\mu, \nu):=\{\pi \in {\mathcal K}(\mu, \nu); (x, \pi_x) \in D(\T), \int_X\T(x, \pi_x)d\mu(x) <+\infty\}\subset {\mathcal K}_{\mathcal A}(\mu, \nu), 
 \]
 we have 
 \[
 \T(\mu, \nu)= \inf\{\int_X c(x, \pi_x)\, d\mu(x); \pi \in {\mathcal K}_{\mathcal \T}(\mu, \nu)\geq {\mcal B}_{c,\A}(\mu, \nu),
 \]
 and we are done. 

If now $D(\T)$ is a transfer set, then 
$
D(\T)=D(\T_r)=\{(\mu, \nu)\in \P(X)\times \P(Y); \mu\prec_\A \nu\}$, and 
\[
 \T(\mu, \nu)= \inf\{\int_X c(x, \pi_x)\, d\mu(x); \pi \in {\mathcal K}_{\mathcal A}(\mu, \nu)\}.
 \]
 The remaining case is similar.
 }

\section{The Kantorovich envelope and the Choquet-Kantorovich capacity}

In this section, we show how the duality between Kantorovich operators and linear transfers give this subclass of Choquet functional capacities a great flexibility for constructing new ones. 

The above observations motivate the search for Kantorovich operators as lower envelopes of arbitrary (standard) maps from $C(Y)$ to $USC(X)$. This will be done in the sequel. We shall make frequent use of the following lemma. 

\begin{lemma} \label{enough1} Let $M$ be a weak$^*$-compact set of $\P(Y)$ and consider the following functional  on  $C(Y)$ 
\eq{\lbl{D}
 P f:=\sup\{\int_Yf d\sigma; \sigma \in \P(Y), \sigma \in M\}. 
}
Then, $P$ is a positively $1$-homogenous subadditive functional from $C(Y)$ to $\R$ and 
\eq{\lbl{enough}
  P f=\sup\{\int_Yf d\sigma; \sigma \in \P(Y),\, \sigma \leq  P\,\,\,  {\rm on}\,\, C_+(Y)\}. 
}
\end{lemma}
\prf{The claimed Kantorovich properties of $ P$ are obvious. Since $P$ is also subadditive, then 
by the Hahn-Banach theorem, we have for any $f\in C(K)$, 
$$ P f=\sup\{\int_Y f\, d\sigma;\,  \sigma \in {\mathcal M}(Y), \sigma \leq  P\,\, {\rm on}\,\, C_+(Y)  \}.$$
If now $f\in C(Y)$, there exists $\sigma \in {\cal M}(Y)$, $\sigma \leq  P $ such that 
$$ P (f+1)=\sup\{\int_Y (f+1)d\tau;\,  \tau \in  {\mathcal M}(Y), \tau \leq P \}= \int_Yf d \sigma+\sigma(Y).$$  
It follows that 
$$1+P f= P (f+1)= \int_Y f\, d\sigma +\sigma (Y) \leq  P f+\sigma (Y)\leq  P f+ P 1\leq P f+1,$$
hence $\sigma (Y)=1$ and $\int_Y f\, d\sigma= P f$. Moreover, since $\sigma \leq P $, we have for any open set $O$, 
$1-\sigma (O)=\sigma (Y\setminus O) \leq  P (Y\setminus O)
$, 
hence $0\leq 1-P (Y\setminus O) \leq \sigma (O)$ and $\sigma$ is therefore a probability measure. Claim (\ref{enough}) follows.
}

\begin{definition} Say that a map $T:C(Y)\to USC(X)$ is {\bf standard} if for every $x\in X$, there is $\nu \in {\cal P}(Y)$ such that 
\begin{equation}\label{dirac_measures}
\sup\limits_{g \in C(Y)}\big\{\int_{Y}g\, d\nu-T g (x)\big\}<+\infty.
\end{equation}
\end{definition}

\begin{theorem}\lbl{envelop1} Let $T:C(Y)\to USC(X)$ be a standard map. Then,
\begin{enumerate}
\item The map $\underbar T$ defined for every $g\in C(Y)$ by the expression 
\begin{equation}\lbl{form1}
{{\underbar T}}\, g(x) :=\sup\limits_{\sigma \in {\cal P}(Y)}\inf\limits_{h\in C(Y)}\{\int_Y (g-h)\, d\sigma + T h(x)\}
\end{equation}
is a Kantorovich operator such that ${{\underbar T}} \leq T$.

\item  For every $g\in C(Y)$, we have 
\begin{equation}\lbl{greatest1}
\hbox{${\underbar T}\, g (x)= \sup\{Sg (x); S$ Kantorovich operator $S\leq T$}\},
\end{equation} 
and $T=\underbar T$ if and only if $T$ is a Kantorovich operator. 
\item If $T$ is positively $1$-homogenous, then  $ \underbar T$ is also positively $1$-homogenous and 
\eq{\lbl{homo}
 \underbar T\, f(x)=\sup\{\int_Y fd\sigma; \sigma \in \P(Y), \sigma \leq T^x \,\, {\rm on}\,\, C_+(Y)\}.
}
\item If $X=Y$ and $T\geq I$, where $I$ is the identity operator, then $ \underbar T\geq I$.  Similarly, if $T\geq T^2$, then $ \underbar T\geq  \underbar T^{\, 2}$.

\end{enumerate}
We shall say that $ \underbar T$ is the {\bf Kantorovich envelope of $T$}.
\end{theorem} 
\prf{
1) Consider the following cost functional $c:X\times \P(Y) \to \R\cup \{+\infty\}$, 
 \begin{equation*}\label{speciale}
c(x, \sigma):= 
\sup_{g \in C(Y)}\big\{\int_{Y}g\, d\sigma-Tg(x)\big\},
\end{equation*}
and note that $c$ is weak$^*$-lower continuous on $ X\times \P(Y)$ and is bounded below since  $T(0) \in USC(X)$, hence  bounded above. Moreover, 
for every $x\in X$,  $c(x, \cdot )$ is convex and by condition (\ref{dirac_measures}) it  is proper. 
Theorem \ref{prop.zero} then yields that the weak optimal transport
\[
{\underline \T}(\mu, \nu):=\inf\{\int_Xc(x, \pi_x) \, d\mu; \pi \in {\mathcal K}(\mu, \nu)\}
\]
is a standard backward linear transfer, with a corresponding  
 backward Kantorovich operator
\eqs{
{{ \underbar T}}\, g(x)=\sup\limits_{\sigma \in \mcal{P}(Y)} \{\int_{Y} gd\sigma - c(x, \sigma)\}=\sup\limits_{\sigma \in {\cal P}(Y)}\inf\limits_{h\in C(Y)}\{\int_Y (g-h)\, d\sigma + T h(x)\}\}.
}
Note that for any $g\in C(Y)$,  \as{
{{ \underbar T}}\, g(x)
 &=\sup_{\sigma \in \mcal{P}(Y)} \inf_{h \in C(Y)} \{\int_{Y} gd\sigma - \int_{Y} hd\sigma + Th(x) \}\\
&\leq \inf_{h \in C(Y)} \sup_{\sigma \in \mcal{P}(Y)} \{\int gd\sigma - \int hd\sigma + Th(x) \}\\
&=\inf_{h \in C(Y)} \{\sup (g -h) + Th(x) \}\\
&\leq Tg (x).
}
2) If now $S$ is another backward Kantorovich operator such that $S\leq T$, then use Sion's min-max principle  \cite{Sion}, that $h\to Sh(x)$ is convex and lower semi-continuous on $C(Y)$ and that $S(h+c)=Sh +c$ whenever $c$ is a constant,  to write 
\as{
{{ \underbar T}}\, g(x) &= \sup_{\sigma \in \mcal{P}(Y)} \inf_{h \in C(Y)} \{\int_{Y} gd\sigma - \int_{Y} hd\sigma + Th(x) \}\\
&\geq \sup_{\sigma \in \mcal{P}(Y)} \inf_{h \in C(Y)}\{\int gd\sigma - \int hd\sigma + Sh(x) \}\\
&= \inf_{h \in C(Y)}\sup_{\sigma \in \mcal{P}(Y)}  \{\int g d\sigma - \int hd\sigma + Sh(x) \}\\
&=\inf\limits _h \{\sup (g -h) + Sh(x) \}\\
&=\inf\limits _h \{S[\sup (g -h) + h](x) \}\\
&\geq Sg (x). 
}
Consider now the operators
\[
\hbox{${S_\infty}f (x)= \sup\{Sf (x); S$ Kantorovich $S\leq T\}$  \quad and \quad ${\hat S}_\infty f (x)=\widehat{S_\infty f}(x)$}, 
\]
where here $\hat g$ is the upper semi-continuous envelope of $g$. 
It is clear that ${{ \underbar T}}\leq S_\infty \leq {\hat S}_\infty \leq T$, and that ${{ \underbar T}}= S_\infty$ since ${\hat S}_\infty$ is a Kantorovich operator.


3) Finally, if $T$ is positively $1$-homogenous, then formula (\ref{form1}) yields immediately that $ \underbar T$ is positively $1$-homogenous. In this case, 
\[
 \underbar T^xg:= \underbar Tg(x)= \underbar T_rg(x)=\sup\{\int_Y fd\sigma; \sigma \in \P(Y), \underline \T (x, \sigma)<+\infty\}.
\]
Use now Lemma \ref{enough1} to write for every $x\in X$, 
\[
 \underbar Tg(x)=\sup\{\int_Y gd\sigma; \sigma \in \P(Y), \sigma \leq  \underbar T^x\}\leq \sup\{\int_Y gd\sigma; \sigma \in \P(Y), \sigma \leq T^x\}.
\]
On the other hand, 
\as{
{{ \underbar T}} g(x)   &=\sup_{\sigma \in \mcal{P}(Y)} \inf_{h \in C(Y)} \{\int_{Y} gd\sigma - \int_{Y} hd\sigma + Th(x) \}\\
&\geq 
\sup_{\sigma \in \mcal{P}(Y), \sigma \leq T^x} \inf_{h \in C(Y)} \{\int_{Y} gd\sigma - \int_{Y} hd\sigma + Th(x) \}\\ 
&\geq \sup\{\int_Y gd\sigma; \sigma \in \P(Y), \sigma \leq T^x\}, 
}
which establishes (\ref{homo}).

4) If $T\geq I$, then $ \underbar T \geq I$ since the identity operator $I$ is clearly a  Kantorovich operator. Now suppose $T^2\leq T$. Since $ \underbar T^2$ is a Kantorovich operator, it suffices to  note that $ \underbar T^2 \leq T^2\leq T$, hence $ \underbar T^2 \leq  \underbar T$.  
}
Dually, we have implicitly shown the existence of an upper linear transfer envelope associated to any standard bounded below, weak$^*$ lower semi-continuous functional  on ${\mathcal P}(X)\times {\mathcal P}(Y)$. Indeed, we have the following.

  \prop{\label{Kant_envelope} 
 Let  ${\mathcal T}:{\mathcal P}(X)\times {\mathcal P}(Y)\to \R \cup\{+\infty\}$ be a standard, convex, bounded below and weak$^*$ lower semi-continuous functional, and let ${\underline \T}$ be the optimal weak transport 
associated to the cost $c(x, \sigma) :=\T(\delta_x, \sigma)$, that is 
$
{\underline \T}(\mu, \nu):=\inf\{\int_Xc(x, \pi_x) \, d\mu; \pi \in {\mathcal K}(\mu, \nu)\}
$. ${\underline \T}$ is then a standard backward linear transfer satisfying  $\T \leq \underline {\cal T}$. Moreover, if $\mcal{S}$ is any standard backward linear transfer  such that $\T \leq \mcal{S}$, then $\underline {\T} \leq \mcal{S}$
}

\prf{Note that by the first part of Theorem \ref{prop.zero}, $\underline {\T}$ is a backward linear transfer with  backward Kantorovich operator
 $T ^-g(x)=\sup_{\sigma \in \mcal{P}(Y)}\{\int_ Yg\, d\sigma-c(x, \sigma)\}.$
To show that $\T \leq {\underline \T}$, note that since $\T$ is jointly convex and lower semi-continuous, then 
 for each $g\in C(Y)$, the functional 
\eqs{
\mu \to (\T_\mu)^*(g) =\sup_{\sigma \in {\mathcal P}(Y)}\{\int_Ygd\sigma -\T(\mu, \sigma) \}
}
is upper semi-continuous and concave. It follows from Jensen's inequality that 
\eqs{
(\T_\mu)^*(g) \geq \int_X(\T_{\delta_x})^*(g) d\mu (x)=\int_X T^-g (x) d\mu (x), 
}
hence 
$\T (\mu, \nu)=(\T_\mu)^{**}(\nu) \leq \sup_{g\in C(Y)}\lf\{\int_Ygd\nu -\int_XT^-g d\mu\rt\}={\underline \T}(\mu, \nu). 
$

  $\underline {\cal T}$ is the smallest backward linear transfer greater than $\T$, since if $\mcal{S}$ is a backward linear transfer and $\T \leq \mcal{S} $, then 
${\underline \T} \leq {\underline {\cal S}}$, and $\underline{\mcal{S}} = \mcal{S}$ by Theorem \ref{envelop1}.
}


  \rem{Suppose ${\mathcal T}$ is any convex, bounded below, and weak$^*$ lower semi-continuous functional on ${\mathcal P}(X)\times {\mathcal P}(Y)$  that is finite on the set of Dirac measures $\{(\delta_x, \delta_y)\,;\, x \in X, y \in Y\}$. One can then define a cost function $c(x,y)=\T(\delta_x, \delta_y)$, and the associated optimal transport $\T_c(\mu,\nu)$. To compare $\T$ with $\T_c$, note that  
\as{
\T_{\delta_x}^*(g) =\sup\{\int_Y g d\nu -{\mathcal T}(\delta_x, \nu); \, \nu\in {\mathcal P}(Y) \} \geq \sup\{g(y) -c(x,y); y\in Y\}=T _c^-g(x)
}
and so 
$\T(\mu, \nu) \leq \underline {\T}(\mu, \nu) \leq {\mathcal T}_c(\mu, \nu).$ 
In many cases, it is not possible to define a proper cost $c(x,y) = \T(\delta_x, \delta_y)$, i.e. $c$ is identically $+\infty$. This is the case for many stochastic transport problems where, for example, transport via Brownian motion makes it impossible for a Dirac measure to be transported to another Dirac measure; see for instance \cite{GKP1}. \qed 
}

\noindent {\bf The Choquet-Kantorovich envelope of a functional capacity}

Given two real-valued set functions $S, T$ on a compact space $Y$, we shall say that
\[\hbox{ 
$S \leq_\K T$ if and only if $S (K)\leq  T(K)$ for any compact set $K\subset Y$.}
\]
For any set function $P$ on $Y$, we follow Choquet and associate the functional 
\eq{\lbl{CE}
\hat P (f)=\int_0^{+\infty}P({f\geq \alpha})\, d\alpha,
} 
on the set of non-negative functions $f$ on $Y$. We call $\hat P$ {\bf the Choquet extension of $P$}. Note that 
\[\hbox{ 
$S \leq_\K T$ if and only if $ \hat S (f) \leq \hat T (f)$ for all $f\in USC_+(Y)$.
} 
\]
\begin{definition} Say that a functional $P:USC(Y) \to \R$ is {\bf saturated} if for any $\sigma \in \P(Y)$, 
\eq{\lbl{s}
\hbox{$\sigma \leq_\K P$ \quad if and only if \quad $\sigma \leq P$\quad on $USC_+(Y)$.}
}
\end{definition}
If $P$ is a capacity, then $\hat P$ is also a capacity that coincides with $P$ on the characteristic functions of sets. $\hat P$ can be seen as an extension of $P$ to non-negative functions in such a way that it is still monotone increasing, satisfies $\hat Pf=\lim_n\uparrow \hat Pf_n$ if $f=\lim_n\uparrow f_n$, and $\hat Pf=\lim_n\downarrow \hat Pf_n$ if $f, f_n$ are upper semi-continuous and $f=\lim_n\downarrow f_n$. We also note that $\hat P$ is positively $1$-homogenous, 
and that it is saturated.  

Similarly to the Choquet functional extension $\hat P$ of $P$, we follows Dellacherie \cite{D} 
and introduce the following  functional extension $\tilde P$ of the set function $P$. 

\begin{definition} Say that a set function $P$ is {\bf common} if there is $\nu \in \P(Y)$ such that $\nu \leq_\K P$.
\end{definition}
For such set functions, we define for $f\in USC(Y)$, 
\eq{\lbl{form3}
\tilde P f=\sup\{\int_Yf d\sigma; \sigma \in \P(Y), \sigma \leq_K P \}.
} 
\begin{lemma} \lbl{sat} Let $P$ be a non-negative common set function. Then, 
\begin{enumerate}
\item $\tilde P$ is a saturated functional on $USC(Y)$.

\item If $P$ originates from a saturated functional on $USC(Y)$, then $\tilde P \leq P$ on $USC_+(Y)$.

\item If $P$ is given on $USC(Y)$ by $Pf=\sup\{\int_Yf d\sigma; \sigma \in M\}$, where $M$ is weak$^*$-compact subset of  $\P(Y)$, then $P \leq \tilde P$ on $USC_+(Y)$, hence if $P$ is also saturated then $P=\tilde P$ on $USC_+(Y)$.

\item We always have $\tilde P=\tilde {\tilde P}$ on $USC_+(Y)$.

\end{enumerate} 
\end{lemma}
\prf{1) That $\tilde P$ is saturated is clear. 2) If now $P$ is saturated, then $\tilde P \leq P$, since
\[
\tilde P f =\sup\{\int_Yf d\sigma; \sigma \in \P(Y), \sigma \leq_K P\}=\sup\{\int_Yf d\sigma; \sigma \in \P(Y), \sigma \leq P\}.
\]

3) If $Pf=\sup\{\int_Yf d\sigma; \sigma \in M\}$, where $M$ is weak$^*$-compact subset of  $\P(Y)$, then Lemma \ref{enough1} gives that 
 \[
 Pf=\sup\{\int_Yf d\sigma; \sigma \in \P(Y), \sigma \leq P\}\leq \sup\{\int_Yf d\sigma; \sigma \in \P(Y), \sigma \leq_\K P\}=\tilde Pf,
 \]
and therefore if $P$ is also saturated, then $P=\tilde P$. 

4) Since $\tilde P$ is saturated, we always have $\tilde {\tilde P}\leq \tilde P$ by the first item.
On the other hand, the set $M=\{\sigma \in \P(Y); \sigma \leq_K P\}$ is equal to $\{\sigma \in \P(Y); \sigma \leq \hat P\}$, where $\hat P$ is the Choquet extension of $P$, hence it is weak$^*$-compact. It follows from item 3) that $\tilde P \leq \tilde {\tilde P}$, hence we have equality. 
}

\begin{definition}\lbl{D} Say that a map $T: F_+(Y) \to F_+(X)$ 
is {\bf common (resp., saturated)} if for every $x\in X$, $T^x$ is common (resp., saturated). 

\end{definition} 
\begin{theorem} Let $T$ be a common functional capacity from $F_+(Y) \to F_+(X)$. Then, 
\begin{enumerate}
\item The map $\tilde T$ defined for every $g\in C(Y)$ by the expression
\eq{\lbl{form3}
\tilde T f (x)=\tilde {T^x} (f)=\sup\{\int_Yf d\sigma; \sigma \in \P(Y), \sigma \leq_K T^x \}
} 
is a saturated positively $1$-homogenous Kantorovich operator such that $\tilde T\leq_\K T$. 
\item If $T$ is a positively $1$-homogenous Kantorovich operator, then $T \leq \tilde T$ on $C_+(Y)$, and 
for any $g\in C(Y)$, 
\begin{equation}\lbl{greatest3}
\hbox{${\tilde T}g (x)= \sup\{Sg (x); S$ is a positively $1$-homogenous Kantorovich operator $\hat S\leq \hat T$}\}.
\end{equation} 

\item If $T$ is saturated, then $\tilde T \leq T$. 

\item If $T$ is a saturated positively $1$-homogenous Kantorovich operator, then $T = \tilde T$ on $C_+(Y)$.

\item  $\tilde T$ is minimal among saturated positively $1$-homogenous Kantorovich operators greater than $T$. 

\end{enumerate}
We shall say that $\tilde T$ is {\bf the Dellacherie envelope of $T$}.
\end{theorem}
\prf{
1) Note first that $\tilde T$ as defined in (\ref{form3}) is clearly a saturated positively $1$-homogenous  operator such that $\tilde T\leq_\K T$. To show that it is a Kantorovich operator, it suffices to prove that the corresponding cost functional 
 \begin{equation}
\tilde c(x, \sigma)=\left\{ \begin{array}{llll}
0 \quad &\hbox{if $\sigma \leq_K T^x$}\\
+\infty \quad &\hbox{\rm otherwise,}
\end{array} \right.
\end{equation} 
is lower semi-continuous. Note that $\tilde c$ is proper since $T$ is common. Let now $x_n\to x$ in $X$ and $\nu_n \to \nu$ weak$^*$ in $\P(Y)$. For any compact set $K\subset Y$ and any $\epsilon>0$, and since 
$T^x$ is a capacity,  there is an open set $O_{2\epsilon} \supset \bar O_\epsilon \supset O_\epsilon \supset K$ such that  $T^x(O_{2\epsilon})-T^x(K) <\epsilon$. Since $\sigma \to \sigma (O_\epsilon)$ is lower semi-continuous, and $x\to T^x(\bar O_\epsilon)$ is upper semi-continuous, we have 
\[
\sigma (K) \leq \sigma (O_\epsilon)\leq \liminf_n\sigma_n(O_\epsilon)\leq \liminf_n T^{x_n}(O_\epsilon)\leq  \limsup_n T^{x_n}(\bar O_\epsilon)\leq T^x(\bar O_\epsilon)\leq T^x(K)+\epsilon.
\]
It follows that $\sigma \leq_K T^x$ and $\tilde c$ is therefore lower semi-continuous. 

2) If now $S$ is positively $1$-homogenous Kantorovich operator such that $S\leq_\K T$, then since 
\[
Sg(x)=S_rg(x)=\sup\{\int_Y g\, d\sigma; \, \T_S(x, \sigma)<+\infty\},
\]
we can apply Lemma \ref{enough1} and write
$$Sg(x)=\sup\{\int_Y gd\sigma;\, \sigma \leq S^x\}\leq \sup\{\int_Y gd\sigma;\, \sigma \leq_\K S^x\}\leq \sup\{\int_Y gd\sigma;\, \sigma \leq_\K T^x\}=\tilde T g (x).$$ 
Formula (\ref{greatest3}) follows along the same line as the proof of (\ref{greatest1})

3) follows from Lemma \ref{sat}. 

4) follows from 2) and 3). 

5) Finally, if $R$ is a saturated positively $1$-homogenous Kantorovich operator such that $R\geq T$, then clearly $R=\tilde R \geq \tilde T$.
}
If now $T:F_+(Y)\to F_+(X)$ is a functional capacity, then 
$$\hat T f(x):=\hat T^x (f)=\int_0^{+\infty}{ T^x}({f\geq \alpha})\, d\alpha,
$$
is a functional capacity that is positively $1$-homogenous and saturated (i.e., for each $x\in X$, $T^x$ is a saturated capacity).   The following notion of saturation is more appropriate for functional capacities. 

\begin{definition} Say that a functional capacity $T:F_+(Y)\to F_+(X)$ is {\bf K-saturated} if for any Kantorovich operator $S: USC(Y) \to USC (X)$ we have that 
\eq{\lbl{K-s}
\hbox{$S \leq_\K T$\quad  if and only if \quad $S\leq T$\quad on $USC_+(Y)$.}
}

Let us also say that a functional capacity $T:F_+(Y)\to F_+(X)$ is {\bf compactly standard} if for every $x\in X$, there is $\nu \in {\cal P}(Y)$ such that 
\eq{\lbl{cond2}
\sup\limits_{K\, {\rm compact\,}}\{\sigma (K)-T\chi_K(x)\} <+\infty.
}
\end{definition}

\begin{theorem}\lbl{C-K} Let $T:F_+(Y)\to F_+(X)$ be a compactly standard functional capacity. Then, 
\begin{enumerate}
\item The map $\bar  T$ defined for every $g\in C(Y)$ by the expression 

\begin{equation}\lbl{form2}
{\bar  T} g(x) :=\sup\limits_{\sigma \in {\cal P}(Y)}\inf\limits_{O\, {\rm open}}\{\int_Y (g-\chi_O)\, d\sigma + T \chi_{\bar O}(x)\},
\end{equation}
is a $K$-saturated Kantorovich operator such that $\overline T \leq_\K  T$. 

\item If $T$ is a Kantorovich operator then $T\leq \overline T$, and for any $g\in C(Y)$, we have 
\begin{equation}\lbl{greatest2}
\hbox{${\bar T}g (x)= \sup\{Sg (x); S$ Kantorovich operator with $\hat S\leq \hat T$}\}.
\end{equation} 

\item If $T$ is $K$-saturated, then $\overline T \leq T$. 

\item If $T$ is a $K$-saturated Kantorovich operator, then $T=\overline T$. 

\item $\overline T$ is minimal among $K$-saturated Kantorovich operators greater than $T$. 

 \end{enumerate}

We shall then say that $\bar  T$ is {\bf the Choquet-Kantorovich} envelope of $T$.
\end{theorem} 
\lem{ Let $T$ be a  functional capacity from $F_+(Y)$ to $F_+(X)$.
\begin{enumerate}
\item  If for some $x\in X$ and $\sigma \in \P(Y)$, $\bar  c(x, \sigma): =\sup\limits_{K \, {\rm compact}\, \subset Y}\{\sigma (K)-T^x(K)\} <+\infty$, then
\eq{\lbl{crucial}
\bar  c (x, \sigma)=\sup\limits_{O \, {\rm open\, in  Y}}\{\sigma (O)-T^x(\bar O)\}.
}
\item If $T$ is a standard functional capacity, then  $(x,\sigma)\to \bar  c(x, \sigma)$ is lower semi-continuous such that for each $x\in X$, $\sigma \to \bar  c(x, \sigma)$ is convex and proper.
 Moreover, 
 \eq{\lbl{crucial.bis}
\bar  c (x, \sigma)\leq \sup\limits_{f\in C(Y)}\{\int_Yfd\sigma -Tf(x)\}.
}
\end{enumerate}
}
\prf{ 1) Note that 
\as{\sup\limits_{O \, {\rm open}}\{\sigma (O)-T^x(\bar O)\}\leq \sup\limits_{O \, {\rm open}}\{\sigma (\bar O)-T^x(\bar O)\}
\leq \sup\limits_{K \, {\rm compact \, in  Y}}\{\sigma (K)-T^x(K)\}.
}
On the other hand, if $\bar  c(x, \sigma)
<+\infty$, then 
 since $T^x$ is a capacity, we have for every $K_\epsilon $ compact such that 
$ \bar  c(x, \sigma) -\epsilon \leq \sigma (K_\epsilon)-T^x(K_\epsilon)\leq \bar  c(x, \sigma),$ 
an open set $O_{2\epsilon} \supset \bar O_\epsilon \supset O_\epsilon \supset K_\epsilon$ such that  
$T^x(O_{2\epsilon})-T^x(K_\epsilon) <\epsilon$, hence 
$$\bar  c(x, \sigma) -2\epsilon \leq \sigma (O_\epsilon)-T^x(O_{2\epsilon} )\leq  \sigma (O_\epsilon)-T^x(\bar O_{\epsilon} )\leq  \sup\limits_{O \, {\rm open}}\{\sigma (O)-T^x(\bar O)\},$$
which implies (\ref{crucial}). 

2) It follows that $\bar  c$ is lower semi-continuous since $\sigma \to \sigma (O)$ is lower semi-continuous for every open set, and $x\to T^x(\bar O)$ is upper semi-continuous since $T$ maps $USC(Y)$ to $USC(X)$ and $\bar O$ is compact. Also, $\sigma \to c(x, \sigma)$ is clearly convex. 

Now again, since $T$ is a functional capacity, we have for every open $O$ such that 
$$ \bar  c(x, \sigma) -\epsilon \leq \sigma (O_\epsilon)-T^x(\bar O_\epsilon)\leq \bar  c(x, \sigma),$$
a continuous function $f_\epsilon \geq \chi_{\bar O_\epsilon}$ such that $T^x(f_\epsilon)-T^x(\bar O_\epsilon) \leq \epsilon$, hence 
$$\bar  c(x, \sigma) -2\epsilon \leq \int_Yf_\epsilon \, d\sigma-T^x(f_\epsilon)\leq  \sup\limits_{f\in C(Y)}\{\int_Yfd\sigma -Tf(x)\},$$
which  implies (\ref{crucial.bis}). 
}
\prf{(of Theorem \ref{C-K}) 
Consider the cost functional $\bar  c:X\times \P(Y) \to \R\cup \{+\infty\}$, 
$$\bar  c(x, \sigma): =\sup\limits_{K \, {\rm compact}}\{\sigma (K)-T^x(K)\},$$ 
and note that condition (\ref{dirac_measures}) ensures that for every $x\in X$, $\bar  c(x, \cdot \, )$ is proper.  
Proposition \ref{Kant_envelope} then yields that the weak transport
$
{\bar  \T}(\mu, \nu):=\inf\{\int_X\bar  c(x, \pi_x) \, d\mu; \pi \in {\mathcal K}(\mu, \nu)\}
$
is a standard backward linear transfer, with a corresponding  
 backward Kantorovich operator
\eqs{
{\bar  T}g(x)=\sup\limits_{\sigma \in \mcal{P}(Y)} \{\int_{Y} gd\sigma - \bar  c(x, \sigma)\}=\sup\limits_{\sigma \in {\cal P}(Y)}\inf\limits_{O\, {\rm open})}\{\int_Y (g-\chi_O)\, d\sigma + T \chi_{\bar O}(x)\}.
}
We now claim that $\bar T$ is $K$-saturated. Indeed, 
if $S$ is a Kantorovich operator such that $\hat S\leq \hat T$, then again use Sion's min-max theorem and the properties of $S$ to write for any $g\in USC(Y)$, 
 \as{
\bar  T g(x)  &=\sup_{\sigma \in \mcal{P}(Y)} \{\int_{Y} gd\sigma - \bar  c (x, \sigma)\}\\
 &= \sup_{\sigma \in \mcal{P}(Y)} \inf\limits_{K\, {\rm compact}}\{\int_{Y} gd\sigma -\sigma ( K)+T^x(K) \}\\
 &\geq \sup_{\sigma \in \mcal{P}(Y)} \inf\limits_{K\, {\rm compact}}\{\int_{Y} gd\sigma -\sigma ( K)+S^x(K) \}\\
   &\geq \sup_{\sigma \in \mcal{P}(Y)} \inf\limits_{f\in C(Y)}\{\int_{Y} gd\sigma -\int_Yf\, d\sigma + Sf (x) \}\\
   &=\inf\limits_{f\in C(Y)}\sup_{\sigma \in \mcal{P}(Y)}\{\int_{Y} gd\sigma -\int_Yf\, d\sigma + Sf (x) \}\\
   &=\inf\limits_{f\in C(Y)}\{\sup_{y\in Y}(g-f)+Sf(x)\}\\
    &=\inf\limits_{f\in C(Y)}\{S[\sup_{y\in Y}(g-f)+f](x)\}\\
&\geq Sg (x). 
}
Write now 
 for any $g\in USC(Y)$, 
\as{
\bar  T g(x)  
  = \sup_{\sigma \in \mcal{P}(Y)} \inf\limits_{K\, {\rm compact}}\{\int_{Y} gd\sigma -\sigma ( K)+T^x(K) \}
\leq \inf\limits_{K\, {\rm compact}} \{\sup_Y (g-\chi_K)+T^x(K) \},
}
which yields that ${{\bar  T}}^x (A) \leq T^x(A)$ for any compact set $A$ in $Y$. 

2) The proof of (\ref{greatest2}) follows the same lines as the proof of (\ref{greatest1}). 

3) If $T$ is $K$-saturated, then $\overline T \leq T$ since $\overline T \leq_\K T$. 

4) follows by combining 2) and 3). 
5) If $R$ is a $K$-saturated Kantorovich operator such that $R\geq T$, then clearly $R=\overline R \geq \overline T$.
}

\defn{A capacity $P$ on a compact space $Y$ is said to be: 
\begin{enumerate}
\item {\em Subadditive of order infinity} if for every compact set $K\subset Y$, any $n\in \N$, and any finite family of compact sets $(K_i)_{i=1}^m$ such that $n\chi_K\leq \Sigma_{i=1}^m\chi_{K_i}$, then
\begin{equation}\lbl{IS}
n P(K)\leq \Sigma_{i=1}^m P(K_i).
\end{equation}

\item {\em Strictly subadditive of order infinity} if for every compact set $K\subset Y$, any $n\in \N$, any $k\in \N\cup \{0\}$, and any finite family of compact sets $(K_i)_{i=1}^m$ such that $k+ n\chi_K\leq \Sigma_{i=1}^m\chi_{K_i}$, then
\begin{equation}\lbl{SIS}
k P(Y)+n P(K)\leq \Sigma_{i=1}^m P(K_i).
\end{equation}

\item {\em Strongly subadditive} if for any two compact subsets $A, B$ of $Y$,
\begin{equation}\lbl{SS}
P(A\cup B)+P(A\cap B)\leq P(A)+ P(B). 
\end{equation}
\end{enumerate}
}
It is known that $(\ref{SS}) \Rightarrow (\ref{SIS}) \Rightarrow (\ref{IS})$ and that 
\begin{itemize}
\item  $P$ is  subadditive of order infinity if $
P(A)=\sup\{\sigma (A); \sigma \in  {\mathcal M}(Y), \sigma \leq_\K P\}$ for 
any compact set $A\subset Y$ (Anger and Lembcke \cite{AL})  .
\item$P$ is strictly subadditive of order infinity if $P(A)=\sup\{\sigma (A); \sigma \in  \P(Y), \sigma \leq_\K P\}$ for any compact set $A\subset Y$  (Anger and Lembcke \cite{AL}) .
 
\item $P$ is strongly subadditive capacity if and only if $\hat P$ is subadditive on $C(Y)$ (Choquet \cite{ Ch1}).
\end{itemize}
If now $T:F_+(Y) \to F_+(X)$ is a functional capacity, then say that $T$ is {\em subadditive of order infinity} (resp., {\em strictly subadditive of order infinity}) (resp., {\em strongly subadditive}), if for any $x\in X$, $T^x$ is a regular capacity that satisfies (\ref{IS}) (resp., (\ref{SIS})), (resp., (\ref{SS})).

\begin{theorem} Let $T:F_+(Y) \to F_+(X)$ be a functional capacity, then 
\begin{enumerate}
\item $T$ is strictly subadditive of order infinity if and only if $\hat T=\hat S$ for some saturated positively $1$-homogenous Kantorovich operator $S$, in which case $\overline T$  is also strictly subadditive of order infinity.

\item $T$ is  strongly subadditive if and only if  $\hat T$ is a Kantorovich operator, in which case $\overline T$  is also strongly subadditive.

\end{enumerate} 

\end{theorem} 
\prf{1) If $T$ is strictly subadditive of order infinity then for any compact set $A$ in $Y$,
\[
T^x(A)=\sup\{\sigma (A); \sigma \in  \P(Y), \sigma \leq_\K  T^x\},
\]
which means that $\tilde T$ and $\bar T$ exist, $\tilde T \leq \bar T$ by the maximality of the latter, and  $\tilde T (\chi_A)=\overline T (\chi_A)=T(\chi_A)$ for any compact $A$, that is, $\hat T=\hat {\tilde T}$. Note that this implies that $\overline T$ is strictly subadditive of order infinity since then,
\as{
\overline T^x(A)&=T^x(A)=\sup\{\sigma (A); \sigma \in  \P(Y), \sigma \leq_\K T^x\}\\
&\geq \sup\{\sigma (A); \sigma \in \P(Y), \sigma \leq_\K \overline T^x\}\\
&\geq \sup\{\sigma (A); \sigma \in \P(Y), \sigma \leq \tilde T^x\}\\
&= \tilde T^x(A)\\
&=\overline T^x(A). 
}
Conversely, if $\hat T=\hat S$ for some positively $1$-homogenous Kantorovich operator $S$, then by the maximality property of $\tilde T$, we have $S\leq \tilde T$. It follows that $\hat T=\hat S \leq  \hat { \tilde T}\leq \hat { \overline T} \leq \hat T$, and therefore we have equality, 
from which follows that for every compact set $A$, 
\[
T^x(A)=\tilde T^x (A)=\sup\{\sigma (A); \sigma \in  \P(Y), \sigma \leq_\K T^x\},
\]
which means that $T$ is strictly subadditive of order infinity,  

2) If now $T$ is strongly subadditive, it is then strictly subadditive of order infinity and 
$\hat T=\hat {\tilde T}=\hat {\overline T}$. Moreover, $\hat T$ is subadditive and therefore $\hat {\tilde T}$ is subadditive, and ${\tilde T}$ is then strongly subadditive. By a remark of Dellacherie \cite{D}, since ${\tilde T}$ is a supremum of measures over a weakly compact subset in $\P(Y)$, ${\tilde T}$ is strongly subadditive if and only if $\hat {\tilde T}=\tilde T$. It follows that $\hat T={\tilde T}$.
Note that this implies that $\overline T$ is strongly subadditive since $\hat {\overline T}=\tilde T$ and the latter is subadditive. 
 
Conversely, assume $\hat T$ is a Kantorovich operator. Since it is also positively $1$-homogenous,  it is subadditive and therefore $T$ is then strongly subadditive by Choquet's criterium. }


\begin{thebibliography}{10}

\bibitem{ABC} J. J. Alibert, G. Bouchitt\'e, T. Champion, {\em A new class of costs for optimal transport planning}, European Journal of Applied Mathematics, Vol 30, 6 (2019) p. 1229-1263.
 
\bibitem{AGL} Alvarez, L., Guichard, F., Lions, P. L., {\em Axioms and fundamental equations of image processing}, Arch. Rational Mech. Anal., Vol. 123, Issue 3 (1993) p. 199-257. 

\bibitem{AL} B. Anger, J. Lembcke, {\em Infinitely Subadditive Capacities as Upper Envelopes of Measures}, Z. Wahrscheinlichkeitstheorie verw. Gebiete 68, 403-414 (1985)

\bibitem{BBP} J. Backhoff-Veraguas,M. Beiglb\"ock, G. Pammer, {\em Existence, duality, and cyclical monotonicity for weak transport costs}, Calculus of Variations and PDE, Vol 58, 6 (2019) p.1-28.

\bibitem{BG1} M. Bowles, N. Ghoussoub, \emph{ A Theory of Transfers: Duality and convolution}, (April 16, 2018, Revised October 24, 2018) 41 pp.   https://arxiv.org/abs/1804.08563
\bibitem{BG2} M. Bowles, N. Ghoussoub, \emph{ Mather Measures and Ergodic Properties of Kantorovich Operators}, (May 7, 2019, Revised June 20, 2019, 2nd revision on October 24, 2019) 109 pp.
https://arxiv.org/abs/1905.05793

\bibitem{BG3} M. Bowles, N. Ghoussoub, \emph{Ergodic properties of Kantorovich operators}, In preparation (2023) 

 \bibitem{D-M} C. Dellacherie, P. A. Meyer, {\em Probabilit\'es et Potentiel}, Chapter X, Hermann (1997).
 
\bibitem{D} C. Dellacherie, \emph{Quelques commentaires sur les prolongements de capacit\'es}, S\'eminaire de probabilit\'es (Strasbourg), tome 5 (1971), p. 77-81.

\bibitem{C} K. J. Ciosmak, {\em Optimal transport and Choquet theory},  arXiv:2001.11292v1 (2020).

\bibitem{Ch1} G. Choquet,  {\em Theory of capacities},  Ann. Inst. Fourier 5, 131-295 (1953/54).

\bibitem{Ch2}
{G. Choquet}
\newblock {\em Lectures on Anaysis}, 
\newblock { W. A. Benjamin, Inc. 1968}.
 
\bibitem{DGKP} S. Dweik, N. Ghoussoub, Y. Kim, A. Palmer, {\em Stochastic Optimal Transport With Free End Time}, Ann. de l'Institut Henri Poincar\'e, Prob. \& Stats, Vol. 57, 2 (2021) p. 700-725.


\bibitem{FS}
W. H. Fleming and H.M. Soner, \emph{Controlled Markov Processes and Viscosity Solutions}, Springer-Verlag, vol. 25, New York, 1993. 
  \bibitem{GLR} I. Gentil, C. Leonard, L. Ripani, {\em About the analogy between optimal transport and minimal entropy,}  Annales de la Facult\'e des Sciences de Toulouse. Math\'ematiques. S\'erie 6, Universit\'e Paul Sabatier 26, 3 (2017),  pp. 569-600.

\bibitem{GKL} N. Ghoussoub, Y. Kim, T. Lim\emph{ Structure of Optimal Martingale Transport in General Dimensions}, Ann. of Probability, Vol. 47, No. 1, (2019 109-164
\bibitem{GKP1} N. Ghoussoub, Y. Kim, A. Palmer, \emph{ A Solution to the Monge Transport Problem for Brownian Martingales}, Ann. of Probability, Vol. 49, 2, (2021) p. 877-907

\bibitem{GKP2} N. Ghoussoub, Y. Kim, A. Palmer, \emph{ PDE Methods For Optimal Skorokhod Embeddings}, Calc. of Variations and PDEs, Calc. Var. 58, 113 (2019) 
\bibitem{GKP3} N. Ghoussoub, Y. Kim, A. Palmer,  {\em Optimal Stopping of Stochastic Transport Minimizing Submartingale Costs}, Trans. AMS, Vol. 374, 10 (2021) p. 6963-6989.
\bibitem{Go} C. Gozlan, P. Roberto, M. Samson, and P. Tetali, \emph{Kantorovich
  duality for general transport costs and applications}, J. of Func. Analysis 273 (2017), no. 11, 3327-3405.
  
  \bibitem{Hl} P. Henry-Labord\`ere, {\em Model-free Hedging: A Martingale Optimal Transport Viewpoint}, Chapman and Hall/CRC (2017) 190 pages.  
\bibitem{Ma}
K.~Marton, \emph{A measure concentration inequality for contracting {M}arkov
  chains}, Geom. Funct. Anal. \textbf{6} (1996), no.~3, 556--571.
  
 
  
  \bibitem{M-T}
{ T. Mikami and M. Thieullen,}
\newblock {\em Duality theorem for the stochastic optimal control problem.}
\newblock Stoch. Process. Appl. 116 (2006), no. 12, 1815--1835

\bibitem{Sion} M. Sion, {\em On general minimax theorems}, Pacific J. Math. 8(1): 171-176 (1958).

  \bibitem{St}
V.~Strassen, \emph{The existence of probability measures with given marginals},
  Ann. Math. Statist. \textbf{36} (1965), 423--439.
\bibitem{Ta}
M.~Talagrand, 
\newblock {\em Transportation cost for gaussian and other product measures,}
\newblock { Geometric and Functional Analysis}, 6:587--600, 1996.
\bibitem{Vi}
C.~Villani, 
\newblock {\em Topics in Optimal Transportation}.
\newblock Graduate Studies in Mathematics 58. American Mathematical Society,
  Providence RI, 2003.
\end{thebibliography}
\end{document}